\documentclass[3p]{elsarticle}

\usepackage{amsmath,amsfonts,amssymb, amsthm}

\usepackage{cases}

\usepackage{graphicx}   
\usepackage{graphics,color}
\usepackage{setspace}  
\usepackage{caption}
\usepackage{listings}
\usepackage{tabularx}
\usepackage{xcolor}
\usepackage{indentfirst}
\usepackage{subcaption}

\usepackage{multirow}
\usepackage{overpic}

\usepackage{verbatim}

\usepackage{enumerate}

\usepackage{bm}

\usepackage[font=footnotesize,labelfont=bf]{caption}

\newtheorem{proposition}{Proposition}[section]

\newtheorem{remark}{Remark}[section]

\usepackage[version = 4]{mhchem}

\newcommand\dd{\mathrm{d}}
\newcommand\pp{\partial}

\newcommand\Vvec{\mathbf{V}}

\newcommand\uvec{\bm{u}}
\newcommand\X{\mathbf{X}}
\newcommand\x{\bm{x}}
\newcommand\y{\bm{y}}
\newcommand\qvec{\bm{q}}
\newcommand\Qvec{\bm{Q}}

\begin{document}

\title{A deterministic--particle--based scheme for micro-macro viscoelastic flows}

\author{Xuelian Bao\footnote{baoxuelian@scut.edu.cn}}
\address{School of Mathematics, South China University of Technology, Guangzhou, Guangdong 510641, People’s Republic of China}

\author{Chun Liu\footnote{cliu124@iit.edu}}
\address{Department of Applied Mathematics, Illinois Institute of Technology, Chicago, IL 60616, United States}

\author{Yiwei Wang\footnote{yiweiw@ucr.edu, corresponding author}}
\address{Department of Mathematics, University of California, Riverside, Riverside, CA 92521, United States}

\begin{abstract}
In this article, we introduce a new method for discretizing micro-macro models of dilute polymeric fluids by integrating a finite element discretization for the macroscopic fluid dynamic equation with a deterministic variational particle scheme for the microscopic Fokker-Planck equation. To address challenges arising from micro-macro coupling, we employ a discrete energetic variational approach to derive a coarse-grained micro-macro model with a particle approximation first and then develop a particle-FEM discretization for the coarse-grained model. The accuracy of the proposed method is evaluated for a Hookean dumbbell model in a Couette flow by comparing the computed velocity field with existing analytical solutions. We also use our method to study nonlinear FENE dumbbell models in different scenarios, such as extension flow, pure shear flow, and lid-driven cavity flow. Numerical examples demonstrate that the proposed deterministic particle approach can accurately capture the various key rheological phenomena in the original FENE model, including hysteresis and $\delta$-function-like spike behavior in extension flows, velocity overshoot phenomenon in pure shear flow, symmetries breaking, vortex center shifting and vortices weakening in the lid-driven cavity flow, with a small number of particles. 

\end{abstract}

\maketitle

\section{Introduction}

Complex fluids comprise a large class of
soft materials, such as polymeric solutions, liquid crystals, ionic solutions, and fiber suspensions. These are fluids with complicated rheological phenomena, arising from different ``elastic'' effects, such as the elasticity of deformable particles, the interaction between charged ions, and bulk elasticity endowed by polymer molecules \cite{liu2009introduction}. Modeling and simulations of complex fluids have been 
interesting problems for a couple of decades \cite{bird1987, larson1998, Bris2012, li2007mathematical}.

Mathematical models of complex fluids are typically categorized as pure macroscopic models \cite{Keunings1997, Lielens1998, Sizaire1999} and micro-macro models \cite{bird1992transport, Bris2012}.  The pure macroscopic models employ an empirical constitutive equation for the stress tensor ${\bm \tau}$ to supplement the conservation laws of mass and momentum \cite{Keunings1997, Lielens1998, Sizaire1999}.
Examples include the Oldroyd-B model \cite{oldroyd1950formulation} and the FENE-P model \cite{peterlin1966hydrodynamics}. 
This approach is advantageous due to its low computational cost, but the closed form of the constitutive equation may fail to capture the intricate flow behaviors of complex fluids, including hysteresis effects in polymeric fluids.
Micro-macro models, on the other hand, couple the macroscopic conservation laws with the microscopic kinetic theory, which describes the origin of the macroscopic stress tensor \cite{Bris2012, bird1992transport}. 
A typical example of a micro-macro model is the dumbbell model of a dilute polymeric fluid \cite{Bris2012, li2007mathematical}.
In this model, a polymer chain is represented by an elastic dumbbell consisting of two beads connected by a single spring. The molecular configuration is characterized by an end-to-end vector of the dumbbell, represented by $\qvec \in \mathbb{R}^d$.
The microscopic dynamic is described by a Fokker-Planck equation for the number density distribution function $f(\x, \  \qvec, \ t)$ with a drift term depending on the macroscopic velocity field $\uvec(\x, t)$. Here, $\x \in \Omega \subset \mathbb{R}^3$ is the macroscopic position, and $\Omega$ is assumed to be a bounded domain with smooth boundary.    
The macroscopic motion of the fluid is described by a Navier--Stokes equation with elastic stress ${\bm \tau}(\x, t)$ induced by the microscopic configuration of polymer chains. 
The corresponding micro-macro model is formulated as follows:
\begin{equation} \label{CPsystemwithf_1}
\begin{cases}
&\rho (\uvec_t + \uvec \cdot \nabla \uvec) + \nabla p= \eta_s \Delta \uvec+\nabla \cdot {\bm \tau}, \quad  {\bm \tau} =  \lambda_p \mathbb{E}_{f} (\nabla_{\bm q} \Psi \otimes \qvec )   = \lambda_p \int_{\mathbb{R}^d} f \nabla_{\qvec} \Psi\otimes \qvec \dd \qvec,   \\
& \nabla\cdot\uvec=0,    \\
& f_t+ \uvec \cdot \nabla f + \nabla_{\qvec}\cdot( (\nabla\uvec) \qvec f)= \frac{2}{\zeta} \nabla_{\qvec}\cdot(f\nabla_{\qvec} \Psi )+ \frac{2k_B T}{\zeta} \Delta_{\qvec} f \ ,   \\
\end{cases}
\end{equation}
subject to a suitable boundary condition on $\uvec$ and $f$. Throughout this paper, $\nabla$ with no indices denotes the derivation with respect to the macroscopic variable $\x$ and $\nabla_{\bm q}$ denotes the derivative with respect to the microscopic variable ${\bm q}$.
Here, $\rho > 0$ is the constant density of the fluid, $\lambda_p > 0$ is a constant that represents the polymer density, $k_B$ is the Boltzmann constant, $T$ is the absolute temperature, $\eta_s > 0$ is the solvent viscosity, $\zeta > 0$ is a constant related to the polymer relaxation time, and $\Psi(\qvec)$ is the spring potential. Typical choices of the elastic potential $\Psi(\qvec)$ include
\begin{equation}\label{Def_Psi}
\Psi(\qvec) = \frac{1}{2} H |\qvec|^2, \quad \text{and} \quad \Psi(\qvec) = \begin{cases} & -\frac{H Q_0^2}{2} \ln \bigg(1-\bigg(\frac{|\qvec|}{Q_0}\bigg)^2\bigg), \quad  |\qvec| < Q_0,  \\
       &   + \infty, \quad \quad \quad  |\qvec| \geq Q_0, \\
 \end{cases}
\end{equation}
known as Hookean and FENE (Finite Extensible Nonlinear Elastic) potentials. Here, $H > 0$ is the elastic constant, and $Q_0$ is the maximum dumbbell extension in FENE models. The interactions among polymer chains are neglected due to the dilute assumption. Alternatively, the microscopic dynamics can be described by a stochastic differential equation (SDE), or a Langevin equation, given by \cite{Bris2012}
\begin{equation}\label{SDE}
\begin{aligned}
\dd \qvec(\x, t)  &= \left( - \uvec \cdot \nabla \qvec(\x, t) + (\nabla \uvec) \qvec(\x, t)   - 2 \zeta^{-1} \nabla_{q} \Psi(\qvec(\x, t))  \right) \dd t + \sqrt{4 k_B T \zeta^{-1}} \dd {\bf W}_t,
\end{aligned}
\end{equation}
where $\dd {\bf W}_t$ is the standard multidimensional white noise.
One of the important properties of the model (\ref{CPsystemwithf_1}) is that the solution satisfies the following energy-dissipation law (see section 2 for details):
\begin{equation}\label{dissipationlaw0}
    \begin{aligned}
    &\frac{\dd}{\dd t} \int_{\Omega}\bigg( \frac{1}{2} \rho |\uvec|^2 + \lambda_p \int_{\mathbb{R}^d}k_B T f \ln f + \Psi f \dd \qvec \bigg) \dd \x = - \int_{\Omega} \bigg(\eta_s |\nabla \uvec|^2  + \int_{\mathbb{R}^d}\frac{\lambda_p \zeta}{2} f | \nabla (k_B T \ln f + \Psi) |^2 \dd \qvec \bigg)\dd \x \ .
    \end{aligned}
\end{equation}
The energy-dissipation law plays a crucial role in understanding the underlying physics of the micro-macro model \cite{Giga2017}, as well as in establishing the well-posedness of the model  \cite{lin-liu-zhang}.

Although micro-macro type models give an elegant description of the origin of the macroscopic stress tensor for various complex fluids \cite{ Bris2012, bird1992transport, lin-liu-zhang}, directly simulating micro-macro models has been a long-standing challenge.
Various computational techniques have been developed to solve the micro-macro model \eqref{CPsystemwithf_1} \cite{weinan2009general, Halin1998, Hulsen1997, Laso1993, Lozinski2011, ottinger1996}. The two main approaches are Langevin-equation-based stochastic simulation methods and direct simulation methods based on the microscopic Fokker-Planck equation \cite{Lozinski2011}. One of the earliest Langevin-equation-based numerical methods is the CONNFFESSIT (Calculation of Non-Newtonian Flow: Finite Elements and Stochastic Simulation Technique) algorithm, which couples a finite element discretization to the macroscopic flow with a numerical solver for the microscopic SDE \eqref{SDE} \cite{Laso1993, ottinger1996}. Along this direction, other stochastic approaches, such as the Lagrangian particle method (LPM) \cite{Halin1998} and the Brownian configuration field (BCF) method \cite{Hulsen1997}, were proposed to reduce the variance and computational cost of the original CONNFFESSIT algorithm. Several extensions and corresponding numerical experiments have been extensively investigated in recent years \cite{Bris2012, Boyaval2010, chauviere2002, Griebel2014, Jourdain2004, Koppol2007, XuSPH2014}. Although stochastic approaches have been the dominant simulation methods for micro-macro models, they suffer from several shortcomings, including high computational costs and stochastic fluctuations. An alternative approach is to simulate the Fokker-Planck equation in the configuration space directly. Examples include Galerkin spectral element technique \cite{Chauviere2004, KnezevicFP, Lozinski2003jcp, shen2012approximation, suen2002} and the lattice Boltzmann technique \cite{Ammar2010, Bergamasco2013}. However, such methods are well suited only for polymeric models having low-dimensional configuration spaces, and the computational cost of Fokker--Planck--based methods increases rapidly for simulations in strong flows (with highly localized distribution function) or involving high-dimensional configuration spaces \cite{Lozinski2011}.

In the context of solving Fokker--Planck type equations, deterministic particle methods have recently gained considerable attention  \cite{russo1990deterministic, degond1990deterministic, lacombe1999presentation, Carrilloblob, wangEVI}. These methods handle the diffusion terms in the equation by using various kernel regularizations \cite{Carrilloblob, degond1990deterministic, lacombe1999presentation, russo1990deterministic}. Unlike Langevin dynamics-based stochastic particle methods, deterministic particle methods often require less number of particles and do not suffer from stochastic fluctuations. The success of deterministic particle methods in solving Fokker--Planck equations has motivated us to develop an efficient numerical method for micro-macro models by incorporating a deterministic particle method. However, the micro-macro coupling in these models presents new challenges that need to be addressed.

The goal of this paper is to develop an efficient numerical method for micro-macro models by utilizing a deterministic particle method for solving the microscopic Fokker-Planck equation.
To overcome these difficulties arising from the micro-macro coupling, we construct the numerical discretiztaion based on the variational formulation of the model. More precisely, we apply the deterministic particle approximation at the energy-dissipation law (\ref{dissipationlaw0}) level and employ an energetic variational approach \cite{wang2020jcp} to derive a coarse-grained micro-macro model with a particle approximation first. A particle-FEM discretization is developed for the coarse-grained model. 
Various numerical experiments have been performed to validate the new scheme via several benchmark problems.
Despite its simplicity, the deterministic particle method is robust, accurate, and capable of catching various rheological behaviors of polymeric fluids. The numerical results obtained by our scheme are in agreement with those from the former work  \cite{Laso1993, Anne1999, XuSPH2014, Hyon2008, Hyon2014}. Moreover, deterministic particle discretization is often more efficient than most stochastic particle methods, where a large ensemble of realizations of the stochastic process and a small step size are needed. 

The rest of this article is organized as follows. 
In section 2, 
we present the basic energy-dissipation law associated with the micro-macro model (\ref{CPsystemwithf_1}) and show that the continuous model can be derived from the energy-dissipation law using an energetic variational approach.
In section 3, we employ the energetic variational approach to derive a coarse-grained particle-based micro-macro model by applying particle approximation at the energy-dissipation law level, and we construct a deterministic particle-FEM scheme for the coarse-grained model.
Various numerical experiments are presented in section 4. 
Finally, the concluding remarks are given in section 5.

\section{Preliminary}

In this section, we briefly introduce the energetic variational approach (EnVarA).
Motivated by non-equilibrium thermodynamics, particularly the celebrated works of Rayleigh \cite{Rayleigh1871} and Onsager \cite{onsager1931reciprocal, onsager1931reciprocal2},
the EnVarA has been successfully applied to build various mathematical models in physics, chemical engineering, and biology \cite{Giga2017, wang2020field}. It also serves as a valuable guideline for developing structure-preserving numerical schemes for these systems \cite{wang2020jcp}.

The key idea of EnVarA is to describe an isothermal and mechanically isolated system by its energy and the rate of energy dissipation over time, along with kinematic (transport) assumptions on the employed variables. 
The dynamics of the system, i.e., the differential equation model, can be derived by combining the Least Action Principle (LAP) and the Maximum Dissipation Principle (MDP) \cite{Giga2017}.
More precisely, according to the first and second laws of thermodynamics \cite{Giga2017, ericksen1992introduction}, an isothermal and closed system possesses an energy-dissipation law
\begin{equation}
\frac{\dd}{\dd t} E^{\text{total}}(t) = - \triangle(t) \leq 0 \ .
\end{equation}
Here, $E^{\text{total}}$ is the total energy, which is the sum of the Helmholtz free energy $\mathcal{F}$ and the kinetic energy $\mathcal{K}$; $\triangle (t) \geq 0$ stands for the rate of energy dissipation, which equals to the rate of entropy production in this case. 
Once these quantities are specified, for the energy part, one can employ the LAP, taking variation of the action functional $\mathcal{A}(\x) = \int_0^T \left( \mathcal{K} - \mathcal{F} \right) \dd t$ with respect to $\x$ (the trajectory in Lagrangian coordinates) \cite{Giga2017, arnol2013mathematical},  to derive the conservative force, i.e.,
$\delta \mathcal{A} =  \int_{0}^T \int_{\Omega} ({\rm force}_{\text{iner}} - {\rm force}_{\text{conv}})\cdot \delta \x  ~\dd \x \dd t.$ For the dissipation part, one can apply the MDP, taking the variation of the Onsager dissipation functional $\mathcal{D}$ with respect to the ``rate'' $\x_t$, to derive the dissipative force, i.e., $\delta \mathcal{D}  = \int_{\Omega} {\rm force}_{\text{diss}} \cdot \delta \x_t~ \dd \x$, where the dissipation functional $\mathcal{D} = \frac{1}{2} \triangle$ in the linear response regime \cite{onsager1931reciprocal}. Consequently, the force balance condition 
results in
\begin{equation}\label{FB}
\frac{\delta \mathcal{A}}{\delta \x} = \frac{\delta \mathcal{D}}{\delta \x_t},
\end{equation}
which is the dynamics of the system. In the case that $\mathcal{K} = 0$, $\frac{\delta \mathcal{A}}{\delta \x} = - \frac{\delta \mathcal{F}}{\delta \x}$, then the dynamics can be written as $\frac{\delta \mathcal{D}}{\delta \x_t} = - \frac{\delta \mathcal{F}}{\delta \x},$ which is a generalized gradient flow.
It is important to notice that the force balance equation (\ref{FB}) uses the strong form of the variational result, as the test functions may be in different spaces in the original variational weak form \cite{Giga2017}. From a modeling perspective, one advantage of utilizing an energy-dissipation law to model a complex system, rather than equations, is that it allows for the systematic inclusion of multiscale and multiphysics coupling and completion.

\subsection{EnVarA for a simple fluid}

To illustrate the idea of EnVarA, we first consider a simple incompressible fluid, which is usually described by an incompressible Navier--Stokes equation
\begin{equation}\label{NS_Incomp}
\begin{cases}
& \rho (\uvec_t + \uvec \cdot \nabla \uvec) + \nabla p = \eta \Delta \uvec, \\
& \nabla \cdot \uvec = 0, \\
& \rho_t + \nabla \cdot (\rho \uvec) = 0 \ , \\
\end{cases}
\end{equation}
with a boundary condition $\uvec = 0$ and initial conditions. Here, $\uvec$ is the fluid velocity, $\uvec \cdot \nabla \uvec = \sum_{j=1}^3 u_j \pp_j u_i$, $\rho$ is the fluid density, $\eta > 0$ is the viscosity, and $p$ is the hydrodynamic pressure. 
Multiplying the first equation of (\ref{NS_Incomp}) by $\uvec$ and using the integration by parts, one can show $\uvec$ satisfies the following energy-dissipation law
\begin{equation}\label{ED_NS}
  \frac{\dd}{\dd t} \int_{\Omega} \frac{1}{2} \rho |\uvec|^2  \dd \x = - \int_{\Omega} \eta |\nabla \uvec|^2 \dd \x\ ,
\end{equation}
where $|\nabla \uvec|$ is the Frobenius norm for matrix $\nabla \uvec$.
Here, the kinetic energy $\mathcal{K}$, the free energy $\mathcal{F}$ and the Onsager dissipation function $\mathcal{D}$ are given by
\begin{equation}
\mathcal{K}  = \int_{\Omega} \frac{1}{2} \rho |\uvec|^2 \dd \x, \quad \mathcal{F} = 0, \quad \mathcal{D} = \int_{\Omega} \frac{1}{2}  \eta |\nabla \uvec|^2 \dd \x\ .
\end{equation}

The goal of EnVarA is to derive the equation of $\uvec$, i.e., the momentum equation in (\ref{NS_Incomp}), from the energy-dissipation law (\ref{ED_NS}), using the LAP and MDP.
To apply the LAP and MDP, one needs to introduce a Lagrangian description of the system by defining the flow map associated with the velocity field $\uvec(\x, t)$ through an ordinary differential equation:
\begin{equation}\label{Def_flow_map}
 \frac{\dd}{\dd t} \x(\X, t) = \uvec(\x(\X, t), t), \quad \x(\X, 0) = \X,
\end{equation} 
where $\X \in \Omega^0$ is the Lagrangian coordinates, and $\Omega^0$ is the reference domain. For fixed $\X$, $\x(\X, t)$ describes the trajectory of a particle initially located at $\X$. For fixed $t$, $\x(\X, t)$ is a diffeomorphism between a reference domain $\Omega^0$ to the current domain $\Omega^t$. In the current case, we have $\Omega^t = \Omega$. It is convenient to define the deformation tensor associated with the flow map $\x(\X, t)$ in both Eulerian and Lagrangian coordinates by
\begin{equation}
{\sf\tilde{F}}(\x(\X,t),t) = 
    {\sf F}(\X, t) = \nabla_{\X} \x(\X, t)\ ,  
\end{equation}
as the deformation tensor ${\sf F}(\X, t)$ carries all the transport information of employed variables \cite{linfh2012}.
Here, the notation $(\nabla_{\X} \x(\X, t))_{ij} =  \frac{\partial x_i}{\partial X_j}$ is used. Applying the chain rule, one can show that in Eulerian coordinates, the deformation tensor ${\sf\tilde{F}}(\x, t)$ satisfies the transport equation  \cite{lin-liu-zhang} 
\begin{equation}\label{transport_F}
  \tilde{\sf F}_t + \uvec\cdot\nabla \tilde{\sf F} = (\nabla \uvec) \tilde{\sf F}\ ,
\end{equation}
where $(\nabla \uvec)_{ij} = u_{i, j} = \frac{\pp u_i}{\pp x_j}$.
For the incompressible fluid, we have
\begin{equation}
 \rho(\x(\X, t), t) = \rho_0(\X), \quad \det {\sf F}(\X, t) = 1,
\end{equation}
where $\rho_0(\X)$ is the initial density. 
Hence, the action $\mathcal{A}[\x]$ can be written as
\begin{equation}
 \mathcal{A} [\x(\X, t)] = \int_{0}^T \mathcal{K} - \mathcal{F} \dd t =  \int_{0}^T \int_{\Omega^0} \rho_0(\X) | \x_t (\X, t)|^2 \dd \X \dd t 
\end{equation}
in Lagrangian coordinates, 
which is a functional of the flow map $\x(\X, t)$.
To compute the variable $\mathcal{A}[\x]$ with respect to $\x(\X, t)$, we consider a perturbation $\x^{\epsilon}(\X, t) = \x(\X, t) + \epsilon \y(\X, t),$ where $\y(\X, t) = \tilde{\y}(\x(\X, t), t)$ is the perturbation satisfying $\tilde{\y} \cdot {\bf n} = 0$ with ${\bf n}$ being the outer normal of $\Omega$. By direct computation, we have
 \begin{equation}
 \begin{aligned}
   \frac{\dd}{\dd \epsilon} \Big|_{\epsilon = 0} \mathcal{A}[\x^{\epsilon}] & = \int_{0}^T \int_{\Omega^0} \rho_0(\X)  x_t (\X, t) \cdot \y_t(\X, t) ~ \dd \X \dd t  = \int_{0}^T \int_{\Omega^0}  - \rho_0(\X)  x_{tt} (\X, t) \cdot \y(\X, t) ~ \dd \X \dd t,
   \end{aligned}
 \end{equation}
where the second identity follows the integration by parts. Pushing forward to Eulerian coordinates, we have
 \begin{equation}\label{LAP_Step1}
     \frac{\dd}{\dd \epsilon} \Big|_{\epsilon = 0} \mathcal{A}[\x^{\epsilon}] = 
     \int_{0}^T \int_{\Omega} - \rho (\uvec_t + \uvec \cdot \nabla \uvec)  \cdot \tilde{\y} ~~ \dd \x \dd t \ ,
 \end{equation}
which indicates
\begin{equation}\label{LAP_Step2}
  \frac{\delta \mathcal{A}}{\delta \x} = - \rho (\uvec_t + \uvec \cdot \nabla \uvec)
\end{equation}
in Eulerian coordinates. For the dissipation part, we apply the MDP by considering a perturbation $\uvec^{\epsilon} (\x, t) = \uvec(\x, t) + \epsilon {\bm v}(\x, t)$. A direct computation shows that
\begin{equation}
  \frac{\dd}{\dd \epsilon} \Big|_{\epsilon = 0} \left(\int_{\Omega} \eta |\nabla \uvec^{\epsilon}|^2 - p (\nabla \cdot \uvec^{\epsilon})    \dd \x \right) = \int_{\Omega}  ( - \eta \Delta u + \nabla p) \cdot {\bm v} \dd \x\ ,
\end{equation}
which indicates 
\begin{equation}\label{MDP_NS}
  \frac{\delta \mathcal{D}}{\delta \x_t}=\frac{\delta \mathcal{D}}{\delta \uvec} = - \eta \Delta u + \nabla p\ .
\end{equation}
Here, $p$ is the Lagrangian multiplier for the incompressible condition $\nabla \cdot \uvec = 0$. Recall the force balance condition (\ref{FB}), we obtain the momentum equation in the incompressible Navier-Stokes equation (\ref{NS_Incomp}) by combining (\ref{LAP_Step2}) and (\ref{MDP_NS}). 

\subsection{EnVarA for the micro-macro model}

As mentioned in the introduction, the well-used micro-macro model (\ref{CPsystemwithf_1}) employs an energy-dissipation law (\ref{dissipationlaw0}). More precisely, by a direct calculation, we can show the following result:
\begin{proposition}
  Suppose that $(\uvec(\x, t), f(\x, \qvec, t))$ is a smooth solution of (\ref{CPsystemwithf_1}) and satisfies the boundary conditions $$\uvec(\x, t) = 0,~~ f(\x, \qvec, t) = 0,  \quad \x \in \partial \Omega, \quad \lim_{\|\qvec\| \rightarrow \infty} f(\x, \qvec, t) = 0, $$ then  $(\uvec(\x, t), f(\x, \qvec, t))$ satisfies the following energy-dissipation law
  \begin{equation}\label{dissipationlaw1}
    \begin{aligned}
    &\frac{\dd}{\dd t} \int_{\Omega}\bigg( \frac{1}{2} \rho |\uvec|^2 + \lambda_p \int_{\mathbb{R}^d}k_B T f \ln f + \Psi f \dd \qvec \bigg) \dd \x = - \int_{\Omega} \bigg(\eta_s |\nabla \uvec|^2 + \int_{\mathbb{R}^d}\frac{\lambda_p \zeta}{2} f  |  {\bf V} -  (\nabla \uvec) \qvec|^2  \dd \qvec \bigg) \dd \x \ ,
    \end{aligned}
    \end{equation}
    where
    \begin{equation}\label{Def_V_micro}
      {\bf V}(\x, \qvec, t) = (\nabla \uvec \qvec) - \frac{2}{\zeta} \nabla_{\qvec} \Psi  - \frac{2k_B T}{\zeta} \nabla_{\qvec} (\ln f)
     \end{equation}
     is the microscopic velocity associated with the microscopic Fokker-Planck equation in (\ref{CPsystemwithf_1}).
\end{proposition}

\begin{proof}
We multiply the first equation of \eqref{CPsystemwithf_1} by $\uvec$ and integrate with respect to $\x$ over domain $\Omega$. Using integration by parts and the fact that $\uvec$ is divergence-free, we have

\begin{equation}\label{ED_macro}
\frac{\dd}{\dd t} \int_{\Omega} \frac{1}{2} \rho |\uvec|^2 \dd \x = - \int_{\Omega}  \eta_s |\nabla \uvec |^2 \dd \x - \int_{\Omega} {\bm \tau} : \nabla \uvec \dd \x\ ,
\end{equation}
where $A:B = \sum_{i, j} A_{ij}B_{ij}$. 
Next, we multiply the third equation of (\ref{CPsystemwithf_1}) by $k_B T \ln f + \Psi$ and integrate with respect to both $\qvec$ over $\mathbb{R}^d$ and $\x$ over $\Omega$. Again, using integration by parts and the fact that $\uvec$ is divergence-free, we can obtain
\begin{equation}\label{ED_micro}
  \frac{\dd}{\dd t} \int_{\Omega} \int_{\mathbb{R}^d} k_B T f \ln f + f \Psi ~ \dd \qvec \dd \x = - \int_{\Omega} \int_{\mathbb{R}^d} \frac{2}{\zeta}  f | \nabla_{\qvec} (k_B T \ln f + \Psi) |^2  + f (\nabla_{\qvec} \Psi \otimes \qvec) : \nabla \uvec ~ \dd \qvec \dd \x . 
\end{equation}
Indeed, due to $\nabla \cdot \uvec = 0$ and $\nabla \Psi(\qvec) = 0$ (as $\Psi$ doesn't depend on $\x$), we have
\begin{equation}
  \begin{aligned}
  &  \int_{\Omega} \int_{\mathbb{R}^d} (\uvec \cdot \nabla f) (k_B T \ln f + \Psi) \dd \qvec \dd \x \\
  & \quad = \int_{\Omega} \int_{\mathbb{R}^d} - (\nabla \cdot \uvec) (k_B T f \ln f + f \Psi) - k_B T \uvec \cdot \nabla f   \dd \qvec \dd \x = \int_{\Omega} \int_{\mathbb{R}^d}  k_B T f (\nabla \cdot \uvec) \dd \qvec \dd \x  = 0 \\
  \end{aligned}
\end{equation}
and
\begin{equation}
  \begin{aligned}
  & \int_{\Omega} \int_{\mathbb{R}^d} \nabla_{\qvec} \cdot ( (\nabla\uvec) \qvec f)  (K_B T \ln f) \dd \qvec \dd \x  =   \int_{\Omega} \int_{\mathbb{R}^d} - k_B T ( (\nabla \uvec) \qvec) \cdot \nabla_{q} f \dd \qvec \dd \x \\
  & =  \int_{\Omega} \int_{\mathbb{R}^d}  k_B T \nabla_{\qvec} \cdot ( (\nabla \uvec) \qvec)  f \dd \qvec \dd \x = \int_{\Omega} \int_{\mathbb{R}^d}  k_B T (\nabla \cdot \uvec)  f \dd \qvec \dd \x  = 0 .\\
  \end{aligned}
\end{equation}

Recall the definition of the microscopic velocity \eqref{Def_V_micro}, the first term on the right-hand side of \eqref{ED_micro} can be written as
\begin{equation}
\frac{2}{\zeta}  f | \nabla_{\qvec} (k_B T \ln f + \Psi) |^2  = \frac{\zeta}{2}  f | {\bm V} - (\nabla \uvec) \qvec |^2 .
\end{equation}
Multiplying (\ref{ED_micro}) by $\lambda_p$ and adding it to (\ref{ED_macro}), we obtain the energy-dissipation law \eqref{dissipationlaw1}.
\end{proof}

In what follows, we show that the micro-macro model (\ref{CPsystemwithf_1}) can be derived from the energy-dissipation law (\ref{dissipationlaw1}) along with the kinematics of the number density distribution function $f(\bm{x}, \bm{q}, t)$:
\begin{equation}\label{fequation}
    \pp_t f + \nabla \cdot (f \uvec) + \nabla_{\bm q} \cdot (f {\bm V}) = 0\ ,
\end{equation}
using the EnVarA.
Here, $\uvec(\x, t)$ is the macroscopic velocity, and ${\bm V}(\x, \qvec, t)$ is microscopic velocity in the configuration space. The goal is to derive the equation of $\uvec(\x, t)$ and ${\bm V}(\x, \qvec, t)$ from (\ref{dissipationlaw1}) by combing LAP and MDP.

We first look at the microscopic dynamics at the configuration space for any given $\x$, which is described by the microscopic energy-dissipation law 
\begin{equation}\label{MicroS_ED}
  \frac{\dd}{\dd t}   \mathcal{F}^{\rm micro}  = - 2 \mathcal{D}^{\rm micro}, \quad 
  \mathcal{F}^{\rm micro} = \int_{\mathbb{R}^d} k_B T f \ln f + \Psi f \dd \qvec, \quad  \mathcal{D}^{\rm micro} = \frac{1}{2}\int_{\mathbb{R}^d} \frac{\zeta}{2} f | {\bm V} - (\nabla \uvec)\qvec|^2 \dd \qvec.
\end{equation}
In the case that $\uvec = 0$, the microscopic energy-dissipation law (\ref{MicroS_ED}) at each $\x$ corresponds to the energy-dissipation law of a linear Fokker--Planck equation \cite{epshteyn2022nonlinear}. Similar to (\ref{Def_flow_map}), we can define a flow map in the configuration space, $\bm{q}(\x, \Qvec, t): \mathbb{R}^d \rightarrow \mathbb{R}^d$, associated with the microscopic velocity $\bm{V}(\x, \qvec, t)$, for any given $\x$, as follows
\begin{equation}
  \frac{\dd}{\dd t} \qvec(\x, \Qvec, t) = {\bm V}(\x, \qvec(\x, \Qvec, t), t), \quad \qvec(\x, \Qvec, 0) = \Qvec, \quad \forall \x. 
\end{equation}
We  refer to  $\qvec(\x, \Qvec, t)$ as the microscopic flow map, to distinguish it from the macroscopic flow map $\x(\X, t)$.
The microscopic dissipation can be interpreted as the relative friction between the microscopic velocity ${\bm V}$ and the macroscopic induced velocity $\widetilde{\mathbf{V}} (\x, \qvec, t)  = (\nabla \uvec) \qvec$.
The macroscopic induced velocity is derived by assuming the effects of macroscopic flow on the microscopic configuration include the linear transport of the center of mass and the stretching of the polymer chain,
i.e., $\qvec(\x, \Qvec, t) = \tilde{\sf F}(\x, t)\Qvec$ \cite{lin-liu-zhang}, which indicates that
$$
\begin{aligned}
\widetilde{\mathbf{V}} (\x, \qvec, t) =\frac{\dd}{\dd t}\bigg(\tilde{\sf F}\Qvec\bigg)=\bigg(\frac{\dd}{\dd t}\tilde{\sf F}\bigg)\Qvec  = (\tilde{\sf F}_t + \uvec\cdot\nabla \tilde{\sf F}) \Qvec =  (\nabla \uvec) \tilde{\sf F} \Qvec = (\nabla \uvec) \qvec\ ,
\end{aligned}
$$
where the transport equation of the deformation tensor (\ref{transport_F}) is used.
The relation $\qvec(\x, \Qvec, t) = \tilde{\sf F}(\x, t)\Qvec$ is the same as the Cauchy--Born rule in solid mechanics that relates the macroscopic deformation of crystals to changes in lattice vectors \cite{cauchyborn2008}.  

To derive the microscopic dynamics, one can apply the EnVarA to the microscopic energy-dissipation law (\ref{MicroS_ED}). 
Since the microscopic dynamics is a gradient flow with $\mathcal{K}^{\rm mirco} = 0$, we only need to compute the variation of the microscopic free energy
\begin{equation}\label{Micro_Action}
\begin{aligned}
    \mathcal{F}^{\rm micro} [\qvec(\x, \Qvec, t)] 
    & =  \int_{\mathbb{R}^d} k_B T f_0 \ln \left(  \frac{f_0(\x, \Qvec)}{\det G(\x, \Qvec, t)} \right)  + \Psi(\qvec) f_0 (\x, \Qvec) \dd \Qvec ,
\end{aligned}
\end{equation}
with respect to the microscopic flow map $\qvec(\x, \Qvec, t)$. Here, $f_0(\x, \Qvec)$ is the initial density  and $G(\x, \Qvec, t) = \nabla_{\Qvec} \qvec(\x, \Qvec, t)$ is the microscopic deformation tensor. Similar to Eqs. \eqref{LAP_Step1} and \eqref{LAP_Step2}, we take the variation of Eq. \eqref{Micro_Action} with respect to $\qvec$ in the microscopic Lagrangian coordinates, and push forward to microscopic Eulerian coordinates, we have
\begin{equation}
\frac{ \delta \mathcal{F}^{\rm micro} [\qvec]}{\delta \qvec} = f \nabla_{\qvec}( k_B T \ln f 
 + \Psi)\ .
\end{equation}
For the dissipation part, by taking variation of $$\mathcal{D}^{\rm micro} = \frac{1}{2}\int_{\mathbb{R}^d} \frac{\zeta}{2} f | {\bm V} - (\nabla \uvec)\qvec|^2 \dd \qvec,$$ we have $\frac{\delta \mathcal{D}^{\rm micro}}{\delta {\bm V}} = \frac{\zeta}{2}({\bf V} - \nabla \uvec \qvec)$.
By the force balance condition (\ref{FB}), 
we obtain 
\begin{equation}\label{EnVarA_V}
\frac{\zeta}{2} f ({\bf V} - \nabla \uvec \qvec) = - f \nabla_{\qvec}( k_B T \ln f + \Psi).
\end{equation}
Combining Eq. \eqref{EnVarA_V} with Eq. \eqref{fequation}, we get the equation on the microscopic scale:
\begin{equation}
f_t+ \nabla\cdot(f \uvec)+\nabla_{\qvec}\cdot(\nabla\uvec \qvec f)= \frac{2}{\zeta} \nabla_{\qvec}\cdot(f\nabla_{\qvec} \Psi )+ \frac{2k_B T}{\zeta} \Delta_{\qvec} f.
\end{equation}

\begin{remark}
    The above derivation is based on the Lagrangian description of microscopic dynamics in the configuration space. Alternatively, the microscopic dynamics can be interpreted as a Wasserstein type gradient flow \cite{jordan1998variational} corresponding to the microscopic free energy in Eulerian coordinates, i.e.,
    \begin{equation}\label{FK_f}
              f_t+ \nabla\cdot(f \uvec)+\nabla_{\qvec}\cdot(\nabla\uvec \qvec f) = \frac{2}{\zeta} \nabla_{\qvec} \cdot (f \nabla_{\qvec} \frac{\delta \mathcal{F}^{\rm micro}}{\delta f}),
    \end{equation}
    where $\mathcal{F}^{\rm micro}$ is the microscopic free energy defined in (\ref{MicroS_ED}), and 
    \begin{equation}
    \frac{\delta \mathcal{F}^{\rm micro}}{\delta f}  = k_B T (\ln f + 1) + \Psi
    \end{equation}
    is the variation of $\mathcal{F}^{\rm micro}$ with respect to the number distribution function $f(\x, \qvec, t)$.
\end{remark}

The variation procedure on the macroscopic scale is similar to that in the simple fluid. To account for the ``separation of scale'', we should treat ${\bm q}$ and ${\bm V}$ as being independent from $\x(\X, t)$ when deriving the macroscopic force balance. The micro-macro coupling is taken into account by the dissipation term $|{\bm V} - (\nabla \uvec) \qvec|^2$. Hence, the action functional is defined by
\begin{equation}
  \begin{aligned}
\mathcal{A}(x) & = \int_{0}^T \int_{\Omega} \left[ \frac{1}{2} \rho |\uvec|^2 - \lambda_p \int_{\mathbb{R}^3} (k_B T f (\ln f - 1) + \Psi(\qvec) f) \dd \qvec \right] \dd \x \dd t, \\
  \end{aligned}
\end{equation}
and the LAP (taking variation of $\mathcal{A}(\x)$ with respect to $\x$) gives rise to
\begin{equation}\label{LAP_macro}
\frac{\delta\mathcal{A}}{\delta \x}= -\rho \x_{tt} = - \rho(\uvec_t + \uvec \cdot \nabla \uvec).
\end{equation}
Meanwhile, for the dissipation part, since $\mathcal{D} = \frac{1}{2} \int_{\Omega} \eta_s |\nabla \uvec|^2 + \mathcal{D}^{\rm micro} ~ \dd \x$, 
the MDP results in
\begin{equation}
\begin{aligned}
  \frac{\delta \mathcal{D}}{\delta \x_t} & =  - \eta_s \Delta \uvec + \frac{\lambda_p \zeta}{2} \nabla \cdot \int f ({\bm V} - \nabla \uvec \qvec) \otimes \qvec \dd \qvec\ ,  \\
      \end{aligned}
\end{equation}
where, $\otimes$ denotes a tensor product and $\uvec\otimes\mathbf{v}$ is a matrix $(u_i v_j)$ for two vectors $\uvec$ and $\mathbf{v}$.
Hence, the force balance results in the equation on the macroscopic scale:
\begin{equation}
\rho(\uvec_t + \uvec \cdot \nabla \uvec) + \nabla p= \eta_s \Delta \uvec +\nabla \cdot {\bm \tau},
\end{equation}
where $p$ is the Lagrangian multiplier of the incompressible condition,
and ${\bm \tau}$ is the induced elastic stress tensor, given by
\begin{equation}\label{stress_IK}
\begin{aligned}
 {\bm \tau} & = \frac{\lambda_p \zeta}{2} \int - f ({\bm V} - \nabla \uvec \qvec) \otimes \qvec \dd \qvec = \lambda_p \int f(\nabla_{\qvec}( k_B T \ln f + \Psi)) \otimes \qvec \dd \qvec \\
 & = \lambda_p \int k_B T \nabla_{\qvec} f \otimes \qvec + f \nabla_{\qvec} \Psi \otimes \qvec \dd \qvec =   \lambda_p \int - k_B T f {\bf I} + f \nabla_{\qvec} \Psi \otimes \qvec \dd \qvec \\
 & =  \lambda_p \left( \int_{\mathbb{R}^d} f \nabla_{\qvec} \Psi \otimes \qvec \dd \qvec - k_B T n {\bf I}  \right).
 \end{aligned}
\end{equation}
The expression (\ref{stress_IK}) is known as the Kramers form of the stress tensor \cite{ottinger1996}.
Since $- k_B T n {\bf I}$ is a diagonal matrix, it can contribute to the pressure term, so that it is convenient to drop it as in (\ref{CPsystemwithf_1}) \cite{ottinger1996}.

\begin{remark}
The induced stress ${\bm \tau}$ can be written as
\begin{equation}\label{IK_f}
{\bm \tau} = \lambda_p \int f(\nabla_{\qvec}( k_B T \ln f + \Psi)) \otimes \qvec \dd \qvec  = \lambda_p \mathbb{E}_f ( (k_B T \nabla_{\qvec} \ln f + \nabla_{\qvec} \Psi ) \otimes  \qvec ) \ ,
\end{equation}
which can be interpreted as the Irving--Kirkwood formula \cite{irving1950statistical}. Here, $k_B T \nabla_{\qvec} \ln f$ can be viewed as an entropic force or Brownian force \cite{bird1985molecular}.
In the case of incompressible flow, this term will contribute to the pressure, as shown in (\ref{stress_IK}).
\end{remark}


\begin{remark}
In the above derivation, the induced stress tensor is derived from the dissipation part in the energy-dissipation law. 
Alternatively,  one should consider  the microscopic configuration to be transported with the flow in the macro-scale,
which indicates that $\qvec = {\sf F} \Qvec$ and $\Vvec = \tilde{\Vvec} = (\nabla \uvec )\qvec$ when performing EnVarA on the macro-scale.
Hence, the action functional $\mathcal{A}[\x]$ can be written as 
\begin{equation}
\mathcal{A}[\x]= \int_{0}^T \int_{\Omega^0}\bigg[ \frac{1}{2} \rho |\x_t|^2 -\lambda_p \int_{\mathbb{R}^d} k_B T f_0 \ln f_0  + \Psi({\sf F} \Qvec) f_0  \dd \Qvec \bigg] \dd \X \dd t
\end{equation}
in Lagrangian coordinates. Here, $f_0(\X, \Qvec)$ is the initial number distribution function. Consider a perturbation $\x^{\epsilon}(\X, t) = \x(\X, t) + \epsilon \y(\X, t),$
where $\y(\X, t) = \tilde{\y}(\x(\X, t), t)$ is the perturbation satisfying $\tilde{\y} \cdot {\bf n} = 0$ with ${\bf n}$ being the outer normal of $\Omega$. Then
\begin{equation*}
  \begin{aligned}
    \frac{\dd}{\dd \epsilon} \mathcal{A} (\x^{\epsilon}) \Big|_{\epsilon = 0} = \int_0^T \int_{\Omega^0} & \Bigl[ - \rho \x_{tt} \cdot \y - \lambda_p \int_{\mathbb{R}^d}  f_0 \nabla_{\qvec} \Psi \otimes \Qvec : \nabla_{\X} \y \dd \Qvec \Bigr] \dd \X \dd t. \\
  \end{aligned}
\end{equation*}
Pushing forward to Eulerian coordinates, we have
\begin{equation*}
  \begin{aligned}
   & \frac{\dd}{\dd \epsilon} \mathcal{A} (\x^{\epsilon}) \Big|_{\epsilon = 0} = \int_0^T \int_{\Omega}  \Bigl[ - \rho (\uvec_t+\uvec \cdot \nabla \uvec) \cdot  \tilde{\y}  - \lambda_p \int_{\mathbb{R}^d} f \nabla_{\qvec} \Psi \otimes \qvec : \nabla_{\x}  \tilde{\y} ) \dd \qvec \Bigr] \dd \x \dd t\\
    & =  \int_0^T \int_{\Omega} \left( - \rho (\uvec_t+\uvec \cdot \nabla \uvec) + \lambda_p \nabla \cdot (\int_{\mathbb{R}^d} f \nabla_{\qvec} \Psi \otimes \qvec \dd \qvec) \right) \cdot \tilde{\y} \dd \x \dd t\ . \\
  \end{aligned}
\end{equation*}
Hence, the LAP leads to 
\begin{equation}\label{LAP_r}
\frac{\delta\mathcal{A}}{\delta\x}= -\rho (\uvec_t+\uvec \cdot \nabla \uvec) +\lambda_p \nabla\cdot\bigg( \int_{\mathbb{R}^d} f \nabla_{\qvec} \Psi\otimes \qvec \dd \qvec \bigg)
\end{equation}
in Eulerian coordinates. 
For the dissipation part,
due to the ``separation of scale'',
the second term in the dissipation in \eqref{dissipationlaw1} vanishes as ${\bm V} = \tilde{ \bm V}$ for the macroscopic dynamics.
Hence, same to (\ref{MDP_NS}), we have
$\frac{\delta\mathcal{D}}{\delta\x_t}= -\eta_s \Delta \uvec + \nabla p$
by employing MDP. Here, $p$ is the Lagrangian multiplier for the incompressible condition $\nabla \cdot \uvec = 0$.
\end{remark}


\subsection{Nondimensionalization of the micro-macro model}

It is convenient to nondimensionalize the micro-macro model by introducing the following  nondimensionalized parameters:
$$
{\rm Re} = \frac{\rho \tilde{U} \tilde{L}}{\eta}, \quad {\rm Wi} = \frac{\lambda \tilde{U}  }{\tilde{L}}, \quad \tilde{\eta}_s = \frac{\eta_s}{\eta}, \quad \epsilon_p = \frac{\eta_p}{\eta}, \quad \lambda = \frac{\zeta}{4H},
$$
where $\tilde{L}= \sqrt{\frac{k_B T}{H}}$ is the  characteristic length scale, $\tilde{U}$ is the characteristic velocity, $\eta_p = \lambda_p k_B T \lambda$ is related to the polymer viscosity, $\eta$ is the total fluid viscosity and $\eta = \eta_s + \eta_p$. The final nondimensionalized system reads as follows, 
\begin{equation}\label{finalmodelnondim_contin}
\left\{
\begin{aligned}
& \mathrm{Re}(\uvec_t+\uvec \cdot \nabla \uvec)+\nabla p= \tilde{\eta}_s \Delta \uvec+\nabla \cdot \bm \tau, \quad {\bm \tau}=\frac{\epsilon_p}{{\rm Wi}} \int f \nabla_{\qvec} \Psi \otimes \qvec \dd \qvec,\\
& \nabla\cdot\uvec=0,\\
& f_t+ \nabla\cdot(\uvec f)+\nabla_{\qvec}\cdot(\nabla\uvec \qvec f)= \frac{1}{2\rm Wi} \nabla_{\qvec}\cdot(f\nabla_{\qvec} \Psi )+  \frac{1}{2\rm Wi} \Delta_{\qvec} f , 
\end{aligned}
\right.
\end{equation}
where $$\Psi(\qvec) = \frac{1}{2}  |\qvec|^2, \qquad  \nabla_{\qvec} \Psi = \qvec,$$ in Hookean model, and $$
\Psi(\qvec) = -\frac{b}{2} \ln (1-|\qvec|^2/b), \qquad  \nabla_{\qvec} \Psi = \frac{\qvec}{1-|\qvec|^2/b}, \quad |\qvec| \leq b$$ with $b = HQ_0^2/k_B T$ in FENE model.

\section{The deterministic particle-FEM method}

In this section, we construct the numerical scheme for the micro-macro model \eqref{CPsystemwithf_1},
which combines a finite element discretization of the macroscopic fluid dynamic equation \cite{bao2021, becker2008,chenrui2015} with a deterministic particle method for the microscopic Fokker-Planck equation \cite{wangEVI}. To overcome the difficulty arising from micro-macro coupling, we first employ a discrete energetic variational approach to derive a particle--based micro--macro model. 
The discrete energetic variational approach follows the idea of ``Approximation-then-Variation'', which first applies particle approximation to the continuous energy dissipation law.
As an advantage, the derived coarse-grained system preserves the variational structure at the particle level.


\subsection{A coarse-grained deterministic particle-based model}

For simplicity, we consider the spatially homogeneous case, and assume the number density function satisfies 
\begin{equation}\label{spatial_hom}
\int_{\mathbb{R}^d} f(\x, \qvec, t) \dd \qvec = 1, \quad \forall \x \in \Omega \ .
\end{equation}

\begin{remark}
The spatial homogeneous assumption is valid if the initial condition $ f_0(\x, \qvec)$ satisfies Eq. (\ref{spatial_hom}). Indeed, recall  Fokker-Planck equation
\begin{equation}\label{fokkerplanck}
f_t+ \nabla\cdot(f \uvec)+\nabla_{\qvec}\cdot(\nabla\uvec \qvec f) = \frac{2}{\zeta} \nabla_{\qvec}\cdot(f\nabla_{\qvec} \Psi )+ \frac{2k_B T}{\zeta} \Delta_{\qvec} f.
\end{equation}
Let $n(\x, t) = \int_{\mathbb{R}^d} f(\x, \qvec, t) \dd \qvec$ be the number density of polymer chains. 
Integrating Eq. (\ref{fokkerplanck}) with respect to $\qvec$ and using the incompressible condition, Eq. \eqref{fokkerplanck} gives 
$$\frac{\pp}{\pp t} n(\x, t) + \uvec \cdot \nabla_{\x} n(\x, t)=0,
$$
which indicates that $n(\x(\X, t), t) = n_0(\X)$ in Lagrangian coordinates. Hence, $n_0(\X) = 1$ leads to $n(\x, t) = 1$.
\end{remark}

The idea of the deterministic particle approximation is to approximate $f(\x, \qvec, t)$ by 
\begin{equation}\label{approxf}
f_N (\x, \qvec, t) =  \frac{1}{N} \sum_{i=1}^{N}  \delta(\qvec - \qvec_i(\x, t)), \quad \forall \x \in \Omega\ ,
\end{equation}
where $N$ is the number of particles at $\x$ and time $t$, $\{\qvec_i (\x, t) \}_{i=1}^N$ represents the set of particles at position $\x$ and time $t$. 
In general, the deterministic particle approximation should be written as  $f_N (\x, \qvec, t) = \sum_{i=1}^{N}  \omega_i(\x, t) \delta(\qvec - \qvec_i(\x, t))$, where $\omega_i(\x, t)$ denotes the weight of $i$-th particle, satisfying $\sum_{i=1}^{N}\omega_i(\x, t) = 1$. In the current study, we only seek for a deterministic particle solution that satisfies $\dot{\omega}_i = 0$ \cite{russo1990deterministic}. Moreover, since $\{ \qvec_i(\x, 0) \}_{i=1}^N$ is sampled from the initial distribution $f_0(\x, \qvec)$, we have $\omega_i(\x, t) = \omega_i(\x, 0) = \frac{1}{N}$.

\begin{remark}
 $\{\qvec_i(\x, t)\}_{i=1}^N$ can be viewed as representative particles that represent information of the number density distribution $f(\x, \qvec, t)$ at $\x$.
 Since only $\{\qvec_i (\x, t) \}_{i=1}^N$  need to be computed at each time-step, the computational cost can be largely reduced, compared to computing $f(\x, \qvec, t)$ directly. 
\end{remark}

One can derive a deterministic particle scheme by directly applying the particle approximation (\ref{approxf}) to the micro-macro model (\ref{CPsystemwithf_1}). For the microscopic Fokker-Planck equation (\ref{fokkerplanck}), recalling the definition of the microscopic velocity ${\bm V}(\x, \qvec, t)$ in (\ref{Def_V_micro}), we have
\begin{equation}\label{eq_using_mu}
\dot{\qvec}_i \approx {\bm V}(\x, \qvec_i, t)  = (\nabla \uvec) \qvec_i - \frac{2}{\zeta} \nabla_{\qvec} \left( \frac{\delta \mathcal{F}^{\rm micro}}{\delta f} \right) \Big|_{\qvec = \qvec_i} = (\nabla \uvec) \qvec_i - \frac{2}{\zeta} \nabla_{\qvec} \Psi (\qvec_i) - \frac{2k_B T}{\zeta}  \nabla_{\qvec} \ln f (\qvec_i) \ .
\end{equation}
The difficulty is how to estimate $\nabla_{\qvec} \ln f(\qvec_i)$ when $f$ is replaced by the empirical measure $f_N$. One approach is to replace the empirical measure $f_N (\qvec) = \frac{1}{N} \sum_{j=1}^N \delta (\qvec - \qvec_j)$ by 
\begin{equation}\label{kernel_smoothing}
f_N^h (\qvec) = f_N * K_h = \frac{1}{N}\sum_{j=1}^N K_h(\qvec, \qvec_j)\ ,
\end{equation}
where $K_{h}(\qvec, \qvec_j)$ is a smooth kernel function and $h$ is the kernel bandwidth \cite{degond1990deterministic}.  A typical choice of $K_{h} (\qvec, \qvec_j)$ is the Gaussian kernel, given by
$$
K_{h}(\qvec, \qvec_j)=\frac{1}{(\sqrt{2\pi}h)^d}\exp\left(-\frac{|\qvec - \qvec_j|^2}{2h^2}\right)\ ,
$$
where $d$ is the dimension of the space. By a direct computation, we have
\begin{equation}
\nabla_{\qvec} \ln f_N^h (\qvec) = \frac{\nabla_{\qvec} f_N^h (\qvec)}{ f_N^h (\qvec)} = \frac{\sum_{j=1}^N \nabla_{\qvec} K_h (\qvec, \qvec_j)}{ \sum_{j=1}^N K_h (\qvec, \qvec_j) }\ ,
\end{equation}
which leads to the deterministic particle equation
\begin{equation}\label{Particle_1}
\dot{\qvec}_i = (\nabla \uvec) \qvec_i - \frac{2}{\zeta} \nabla_{\qvec} \Psi (\qvec_i) -  \frac{2k_B T}{\zeta}  \frac{ \sum_{j=1}^N  \nabla_{\qvec_i} K_h(\qvec_i, \qvec_j) }{\sum_{j=1}^{N} K_h(\qvec_i, \qvec_j)}. 
\end{equation}
On the macroscopic scale, the deterministic particle approximation leads to an approximated stress tensor
\begin{equation}\label{directapproxitau}
 {\bm \tau} \approx  \lambda_p \int_{\mathbb{R}^d} f_N \nabla_{\qvec} \Psi \otimes \qvec \dd \qvec = \frac{\lambda_p}{N} \sum_{i=1}^N \nabla_{\qvec} \Psi (\qvec_i) \otimes \qvec_i
\end{equation}
by replacing $f$ by $f_N$.
However, 
the above approximation fails to maintain the variational structure, i.e., admits an energy-dissipation law as a counterpart of (\ref{dissipationlaw1}).

To maintain the variational structure in the particle level, we adopt the idea of ``approximation-then-variation'', which first substitutes the approximation \eqref{approxf} into the energy-dissipation law \eqref{dissipationlaw1}. 
To derive the equation of $\qvec_i$, we first substitute the approximation \eqref{approxf} into the microscopic energy-dissipation law \eqref{MicroS_ED}. Since the term $\ln f_N(\x, \qvec, t)$ can not be defined in a proper way, we introduce the kernel smoothing (\ref{kernel_smoothing}) to handle the $\ln f$ term. More precisely, we first define a regularized microscopic free energy \cite{Carrilloblob}:
\begin{equation}\label{F_micro_q}
\begin{aligned}
\mathcal{F}_h^{\rm micro}[f] & =  \int_{\mathbb{R}^d} k_B T f \ln (f * K_h) + \Psi f \dd \qvec\ ,
\end{aligned}
\end{equation}
where $K_h$ is a kernel function introduced before. Replacing $f$ with $f_N$ in (\ref{F_micro_q}), we end up with a microscopic free energy in terms of particles $\{ \qvec_i \}_{i=1}^N$, given by
\begin{equation}
\mathcal{F}^{\rm micro}_{N, h}(\qvec_1, \ldots, \qvec_N) =  \frac{1}{N} \sum_{i=1}^N \left[ k_B T \ln \left(\frac{1}{N} \sum_{j=1}^N K_h(\qvec_i, \qvec_j )\right)+\Psi(\qvec_i)\right].
\end{equation}
The microscopic dissipation potential $\mathcal{D}^{\rm micro}$ can be approximated by
\begin{equation}\label{D_micro_q}
\begin{aligned}
&\mathcal{D}_{N, h}^{\rm micro}  = \frac{1}{2} \frac{\zeta}{2}\frac{1}{N} \sum_{i=1}^N |\dot{\qvec}_i - \nabla \uvec \qvec_i|^2\ ,
\end{aligned}
\end{equation}
where the approximation $V(\x, \qvec_i, t) \approx \dot{\qvec_i}$ is used, and  $\dot{\qvec_i} = \pp_t \qvec_i + \uvec \cdot \nabla \qvec_i$ is the material derivative of $\qvec_i$. One can view $(\nabla \uvec) \qvec_i$ as particle velocity that is induced by the macroscopic flow as in the continuous model. 

Within the microscopic free energy (\ref{F_micro_q}) and dissipation (\ref{D_micro_q}) in the particle level, the dynamics of $\qvec_i(\x, t)$ can be derived by performing the EnVarA in terms of $\qvec_i$ and $\dot{\qvec_i}$, i.e.,
$$\frac{\delta \mathcal{D}_{N, h}^{\rm micro}}{\delta \dot{\qvec}_i}=-\frac{\delta \mathcal{F}_{N, h}^{\rm micro}}{\delta \qvec_i}\ .$$
By a direct computation, we have
\begin{equation}\label{Eq_qi_1}
\frac{\zeta}{2}\frac{1}{N} (\dot{\qvec}_i - \nabla \uvec \qvec_i) = - {\bm \mu}_i\ ,
\end{equation}
where 
\begin{equation}
{\bm \mu}_i = \frac{1}{N} \bigg[ k_B T \bigg(\frac{\sum_{j=1}^N  \nabla_{\qvec_i} K_h(\qvec_i, \qvec_j)}{\sum_{j = 1}^N K_h(\qvec_i, \qvec_j) }  + \sum_{k=1}^N \frac{\nabla_{\qvec_i} K_h(\qvec_k, \qvec_i )}{\sum_{j=1}^N K_h(\qvec_k, \qvec_j)}\bigg) + \nabla_{\qvec_i} \Psi(\qvec_i) \bigg].
\end{equation}
The equation (\ref{Eq_qi_1}) can be rewritten as
\begin{equation}\label{Eq_qi}
\begin{aligned}
  \pp_t \qvec_i + \uvec \cdot \nabla \qvec_i & = (\nabla \uvec)\qvec_i - \frac{2}{\zeta} \bigg[ k_B T \bigg(\frac{\sum_{j=1}^N  \nabla_{\qvec_i} K_h(\qvec_i, \qvec_j)}{\sum_{j = 1}^N K_h(\qvec_i, \qvec_j) }  + \sum_{k=1}^N \frac{\nabla_{\qvec_i} K_h(\qvec_k, \qvec_i )}{\sum_{j=1}^N K_h(\qvec_k, \qvec_j)}\bigg) + \nabla_{\qvec_i} \Psi(\qvec_i) \bigg].
\end{aligned}
\end{equation}
It can be noticed that Eq. (\ref{Eq_qi}) is a gradient flow with respect to $ \qvec_i $ in the absence of the flow, i.e. $\uvec = 0$ at $\forall \x$. Comparing (\ref{Eq_qi}) with (\ref{Particle_1}), we notice an additional term, crucial for preserving the variational structure at the particle level, which can only be derived through the  "Approximation-then-Variation" approach.

\begin{remark}
   The same deterministic particle scheme can be obtained by replacing $\mathcal{F}^{\rm micro}$ in (\ref{FK_f}) by $\mathcal{F}^{\rm micro}_h$ \cite{reich2021fokker}.
   Through direct calculation, we have
       \begin{equation}\label{dF_h_df}
       \begin{aligned}
         \frac{\delta \mathcal{F}_h^{\rm micro}}{\delta f}  =  \left(  k_B T \ln (K_h * f) + k_B T  \left( \frac{K_h}{K_h * f} \right) * f + \Psi  \right) .\\
        \end{aligned}
       \end{equation}
      Consequently,  the microscopic velocity equation becomes
      \begin{equation}\label{V_with_F_h}
       f {\bm V}(\x, \qvec, t) = f (\nabla \uvec) \qvec - \frac{2}{\zeta} (  f \nabla_{\qvec}  \frac{\delta \mathcal{F}_h^{\rm micro}}{\delta f}  ).
      \end{equation}
      The kernel regularized microscopic free energy $\mathcal{F}_h^{\rm micro}$ permits a direct substitution of $f$ with $f_N$ in Eq. (\ref{V_with_F_h}). In other word, (\ref{V_with_F_h}) has the property that "particles remain particles" \cite{Carrilloblob}.
      By direct calculation, we have
      \begin{equation}
          V(\x, \qvec_i, t) = (\nabla \uvec) \qvec_i  - \frac{2}{\zeta} \nabla_{\qvec} \left( k_B T \ln \left(  \frac{1}{N} \sum_{j=1}^N K_h (\qvec, \qvec_j) \right) + k_B T   \sum_{k=1}^N \frac{  K_h (\qvec, \qvec_k)}{  \sum_{j=1}^N  K (\qvec_k, \qvec_j) } + \Psi(\qvec) \right) \Bigg|_{\qvec = \qvec_i}\ ,
      \end{equation}
      which is equivalent to (\ref{Eq_qi}).
\end{remark}

The variational procedure for the macroscopic flow is almost the same as that in the continuous case, as shown in section 2. By applying the LAP, we obtained (\ref{LAP_macro}), while the MDP leads to
\begin{equation}
  \frac{\delta \mathcal{D}_N}{\delta \x_t}  =  - \eta_s \Delta \uvec + \frac{\lambda_p \zeta}{2}   \nabla \cdot  
   \left( \frac{1}{N} \sum_{i=1}^N ( \dot{\qvec_i} - \nabla \uvec \qvec_i) \otimes \qvec_i  \right) ,
\end{equation}
where $\mathcal{D}_{N} = \frac{1}{2} \int_{\Omega} \eta_s |\nabla \uvec|^2 + \lambda_p \mathcal{D}^{\rm micro}_N  \dd \x$ is the macroscopic dissipation. 
Hence, we can define the stress as
\begin{equation}
{\bm \tau} =  - \frac{\lambda_p \zeta}{2} \frac{1}{N} \sum_{i=1}^N ( \dot{\qvec_i} - \nabla \uvec \qvec_i) \otimes \qvec_i  = \lambda_p \sum_{i=1}^N   {\bm \mu}_i \otimes \qvec_i. 
\end{equation}
The final micro-macro system with particle approximation is given by
\begin{equation}\label{finalmodel}
\left\{
\begin{aligned}
&\rho(\uvec_t+\uvec \cdot \nabla \uvec) + \nabla p= \eta_s \Delta \uvec+\nabla \cdot {\bm \tau}, \quad {\bm \tau}(\x, t) = \lambda_p  \sum_{i=1}^N  {\bm \mu}_i \otimes \qvec_i,\\
& \nabla\cdot\uvec=0,\\
\end{aligned}
\right.
\end{equation}
where $\qvec_i(\x, t)$ satisfies Eq. \eqref{Eq_qi}.

\begin{remark}
One can notice that the stress tensor obtained through the energetic variational approach differs from that obtained by directly replacing $f$ with $f_N$ in the stress in Eq. (\ref{CPsystemwithf_1}). The main reason is that the Brownian force in the particle approximation $\nabla_{\qvec_i} \ln f_N^{h}$ can no longer be written in a diagonal matrix form as in (\ref{stress_IK}). Therefore, it does not vanish as it does in the continuous model.
\end{remark}

One can view the macroscopic flow equation (\ref{finalmodel}) along with the microscopic evolution equation (\ref{Eq_qi}) as a coarse-grained model for the original micro-macro model (\ref{CPsystemwithf_1}). As an advantage of the energetic variational approach, one can maintain the variational structure in the particle level,  which is crucial in establishing the well-posedness of the coarse-grained model and proving the convergence for $N \rightarrow \infty$. More precisely, we have the following result:
\begin{proposition}
The coarse-grained model \eqref{Eq_qi_1}-\eqref{finalmodel} satisfies the energy-dissipation law
\begin{equation}\label{ED_particle}
  \frac{\dd}{\dd t} \int_{\Omega} \frac{1}{2} \rho |\uvec|^2 +\frac{ \lambda_p }{N} \sum_{i=1}^N \left[ k_B T \ln \left(\frac{1}{N} \sum_{j=1}^N K_h(\qvec_i - \qvec_j )\right)+\Psi(\qvec_i)\right]  \dd \x = - \int_{\Omega} \eta_s |\nabla \uvec|^2 +  \frac{\lambda_p \zeta}{2}\frac{1}{N} \sum_{i=1}^N |\dot{\qvec}_i - \nabla \uvec \qvec_i|^2 \dd \x. 
\end{equation}

\end{proposition}

\begin{proof}
The proof is similar to that in the continuous case. 
Let $F(\qvec_1, \ldots \qvec_N) =  \mathcal{F}^{\rm micro}_{N, h}(\qvec_1, \ldots, \qvec_N)$.
Multiplying (\ref{Eq_qi_1}) by ${\bm \mu}_i$, summing with respect to $i$, and integrating over the domain, we have
\begin{equation}\label{ED_micro_1}
\begin{aligned}
 \frac{\dd}{\dd t} \int_{\Omega} F(\qvec_1, \ldots \qvec_N) \dd \x & = \int_{\Omega} \sum_{i=1}^N ((\nabla \uvec) \qvec_i) \cdot {\bm \mu}_i -  \sum_{i=1}^N \frac{2 N}{\zeta}|{\bm \mu}_i|^2 \dd \x \\
 & = \int_{\Omega}  ( \sum_{i=1}^N {\bm \mu}_i \otimes \qvec_i ) : (\nabla \uvec)  -  \frac{\zeta}{2}\frac{1}{N} \sum_{i=1}^N |\dot{\qvec}_i - \nabla \uvec \qvec_i|^2 \dd \x\ ,
 \end{aligned}
\end{equation}
due to the fact that 
\begin{equation}
 \sum_{i=1}^N \int_{\Omega} (\uvec \cdot \nabla \qvec_i) \cdot {\bm \mu}_i  \dd \x  = \int_{\Omega} \uvec  \cdot \nabla F(\qvec_1, \ldots \qvec_N) = - \int_{\Omega} (\nabla \cdot \uvec)  F(\qvec_1, \ldots \qvec_N)  \dd \x = 0.
\end{equation}

Similar to proposition 2.1, we multiply (\ref{finalmodel}) by $\uvec$ and integrate with respect to $\x$ over domain $\Omega$. Using
integration by parts and the fact that $\uvec$ is divergence-free, we have
\begin{equation}\label{ED_macro_p_1}
 \frac{\dd}{\dd t} \int_{\Omega} \frac{1}{2} \rho |\uvec|^2 \dd \x  = - \int_{\Omega} \eta_s |\nabla \uvec|^2 - {\bm \tau} : \nabla \uvec  \dd \x. 
\end{equation}

By multiplying (\ref{ED_micro_1}) by $\lambda_p$ and adding the result to (\ref{ED_macro_p_1}), we obtain the energy-dissipation law (\ref{ED_particle}) at the particle level.

\end{proof}

The coarse-grained model \eqref{Eq_qi}-\eqref{finalmodel} can be non-dimensionalized by using the same nondimensionalized parameters as in the continuous case. The final nondimensionalized system reads
\begin{equation}\label{finalmodelnondim}
\left\{
\begin{aligned}
& \mathrm{Re}(\uvec_t+\uvec \cdot \nabla \uvec)+\nabla p = \tilde{\eta}_s \Delta \uvec+\nabla \cdot \bm \tau, \quad {\bm \tau} =  \frac{\epsilon_p}{\rm Wi} \sum_{i=1}^N  {\bm \mu}_i \otimes \qvec_i, \\
& {\bm \mu}_i =  \frac{1}{N} \bigg( \frac{\sum_{j=1}^N  \nabla_{\qvec_i} K_h(\qvec_i, \qvec_j)}{\sum_{j = 1}^N K_h(\qvec_i, \qvec_j) }  + \sum_{k=1}^N \frac{\nabla_{\qvec_i} K_h(\qvec_k, \qvec_i )}{\sum_{j=1}^N K_h(\qvec_k, \qvec_j)}
+\nabla_{\qvec_i} \Psi(\qvec_i) \bigg), \\
& \nabla\cdot\uvec=0,
\end{aligned}
\right.
\end{equation}
with $\qvec_i(\x, t)$ satisfying
\begin{equation}\label{Eq_qinondim}
\begin{aligned}
\pp_t \qvec_i + \uvec \cdot \nabla \qvec_i -(\nabla \uvec)\qvec_i &=- \frac{1}{2 {\rm Wi}} \bigg( \frac{\sum_{j=1}^N  \nabla_{\qvec_i} K_h(\qvec_i, \qvec_j)}{\sum_{j = 1}^N K_h(\qvec_i, \qvec_j) } + \sum_{k=1}^N \frac{\nabla_{\qvec_i} K_h(\qvec_k, \qvec_i)}{\sum_{j=1}^N K_h(\qvec_k, \qvec_j)}+ \nabla_{\qvec_i} \Psi(\qvec_i) \bigg).
\end{aligned}
\end{equation}
Due to the presence of the convection term $\uvec \cdot \nabla \qvec_i(\x, t)$, $\qvec_i(\x, t)$ should be viewed as a field rather than a particle at $\x$. However, introducing a spatial discretization might significantly increase the computational cost. To overcome this difficulty, we will adopt a Lagrangian approach to deal with the convection term, which leads to an independent ensemble of particles at each $\x$ when $\uvec \cdot \nabla \qvec_i(\x, t) \neq 0$.


\begin{remark}
A key step in the above deterministic particle scheme is to replace the empirical measure $f_N$ by $f_N^h$ using kernel smoothing (\ref{kernel_smoothing}) when computing $\ln f$ in the free energy. The choice of kernel bandwidth $h$ is important for the accuracy and robustness of the numerical scheme below. 
Intuitively, if $h$ is too small, then $f_N^h$ will be oscillated. Consequently, it fails to provide a good approximation to $\ln f$ and $\nabla (\ln f)$, and the term $\frac{\sum_{j=1}^N  \nabla_{\qvec_i} K_h(\qvec_i, \qvec_j)}{\sum_{j = 1}^N K_h(\qvec_i, \qvec_j) } + \sum_{k=1}^N \frac{\nabla_{\qvec_i} K_h(\qvec_k, \qvec_i)}{\sum_{j=1}^N K_h(\qvec_k, \qvec_j)}$ will be very large. On the other hand, if $h$ is too large, $f_N^h$ will be flatten and the term $\frac{\sum_{j=1}^N  \nabla_{\qvec_i} K_h(\qvec_i, \qvec_j)}{\sum_{j = 1}^N K_h(\qvec_i, \qvec_j) } + \sum_{k=1}^N \frac{\nabla_{\qvec_i} K_h(\qvec_k, \qvec_i)}{\sum_{j=1}^N K_h(\qvec_k, \qvec_j)}$ will be almost zero, which also fails to approximate $\nabla \ln f_N^h$.
The optimal kernel bandwidth depends on the potential $\Psi(\qvec)$ and the macroscopic flow. In the current study, we choose the kernel bandwidth $h$ through multiple numerical experiments (see the numerical sections for details).


\end{remark}

\subsection{Full discrete scheme}
In this subsection, we construct a full discrete scheme for the coarse-grained model (Eqs. \eqref{finalmodelnondim} and \eqref{Eq_qinondim}).
To solve the micro-macro system numerically, it is a natural idea to develop some decoupled schemes. Precisely, we propose the following scheme for the temporal discretization:
\begin{itemize}
    \item Step 1: Treat the viscoelastic stress ${\bm \tau}^n$ explicitly, and solve the equation \eqref{finalmodelnondim} to obtain updated values for the velocity and pressure. 
   \item Step 2: Use the updated velocity field $\uvec^{n+1}$ to solve the equation of $\qvec_i$ at each node, and then update values of the viscoelastic stress, denoted by ${\bm \tau}^{n+1}$.
   \end{itemize}


Eq. \eqref{finalmodelnondim} in the first step can be solved by a standard incremental pressure-correction scheme \cite{Guermond2006} reads as follows: 
   \begin{itemize}
      \item Step 1.1: 
\begin{equation}\label{Eq_u_n}
\begin{aligned}
  & {\rm Re} \bigg(\frac{\tilde{\uvec}^{n+1} - \uvec^n}{\Delta t} + \uvec^n \cdot \nabla \tilde{\uvec}^{n+1}\bigg) + \nabla p^{n} = \tilde{\eta}_s \Delta \tilde{\uvec}^{n+1} + \nabla \cdot \bm\tau^n, \\
\end{aligned}
\end{equation}
\item Step 1.2: 
 \begin{equation} 
 \left\{
 \begin{aligned}
  & \frac{\uvec^{n+1} - \tilde{\uvec}^{n+1}}{\Delta t} + \nabla (p^{n+1} - p^n) =  0, \\ 
  & \nabla \cdot \uvec^{n+1} = 0 \ .  \\
\end{aligned}
\right.
\end{equation}
\end{itemize}

 We use the finite element method developed in Refs. \cite{becker2008, chenrui2015} for spatial discretization, employing the inf-sup stable Iso-P2/P1 element \cite{brezzi91, Tabata00} for velocity and pressure, and a linear element for each stress component. 
 More precisely, let $\Omega$ be the bounded computational domain, $\mathcal{T}_h$ and $\hat{\mathcal{T}}_{h}$ be two triangulations of $\Omega$, with $\mathcal{T}_h$ being the uniform refinement of $\hat{\mathcal{T}}_h$. We denote $\mathcal{T}_h$ and $\hat{\mathcal{T}}_h$ as sets of simplexes $\{\kappa_e | e=1,\dots,M\}$ and 
 $\{\hat{\kappa}_e | e=1,\dots,\hat{M}\}$, respectively. $N_h = \{\x_1,\dots,\x_{N_x}\}$ and $\hat{N}_h = 
 \{\hat{\x}_1,\dots,\hat{\x}_{\hat{N}_x}\}$ are sets of nodal points. We construct the finite-dimensional subspaces $S_h$, $\hat{S}_h\subset H^1(\Omega)$ and $S_h^0\subset H_0^1(\Omega)$ as follows:
$$
S_h = \{g \in C^0(\overline{\Omega}): g|_{\kappa} \in P_1(\kappa)\}, \quad  \hat{S}_h = \{g \in C^0(\overline{\Omega}): g|_{\hat{\kappa}} \in P_1(\hat{\kappa})\}, \quad S_h^0 = \{g \in S_h: g|_{\partial \Omega}=0 \},
$$
where $P_r(\kappa)$ is the space of polynomial functions of degree less than or equal to $r$ on the simplex $\kappa$. We let 
$V_{{\bm \tau}} = (S_h)^{d^2}$ with $d$ the dimension of space,  $V_{\uvec_h} = (S_h^0)^d$ and $M_h = \hat{S}_h \cap L_0^2(\Omega)$. One can show that $V_{\uvec_h}$ and $M_h$ satisfy the inf–sup condition \cite{mixedFEM, brezzi91}
$$
\inf_{p_h\in M_h} \sup_{\uvec_h \in V_{\uvec_h}} \frac{\int_{\Omega} p_h \nabla \cdot \uvec_h d\x}{\| p_h\| \| \uvec_h\|_1} \geq C,
$$
where $C>0$ is independent of mesh size $h$ and $\| \uvec_h\|_1 = \|\nabla \uvec_h \| + \| \uvec_h \|$.

The full discretization scheme for Step 1 can be summarized as follows.
Given 
$\uvec_h^n \in V_{\uvec_h}$, $\bm \tau_h^n \in V_{\bm \tau}$ and $p_h^n \in M_h$ for $n>0$, we compute $\uvec_h^{n+1}$ and $p_h^{n+1}$ by the following algorithm:
    \begin{itemize}
      \item Step 1.1: Find $\tilde{\uvec}_h^{n+1} \in V_{\uvec_h}$, such that for any $\mathbf{v} \in V_{\uvec_h}$, 
  \begin{equation*}
  \begin{aligned}
   &{\rm Re} \bigg(\frac{1}{\Delta t} \tilde{\uvec}_h^{n+1}+\uvec_h^n \cdot \nabla \tilde{\uvec}_h^{n+1}, \mathbf{v}\bigg)+(\tilde{\eta}_s \nabla \tilde{\uvec}_h^{n+1}, \nabla\mathbf{v})  =  {\rm Re} \bigg(\frac{1}{\Delta t} \uvec_h^{n}, \mathbf{v}\bigg)-(\nabla p_h^n, \mathbf{v}) + (\nabla \cdot \bm \tau_h^n, \mathbf{v}).
  \end{aligned}
  \end{equation*} \\
      \item Step 1.2: Find $p_h^{n+1} \in M_h$, such that for any $\psi \in M_h$, 
  \begin{equation*}
      \begin{aligned}
      & (\nabla (p_h^{n+1} - p_h^n), \nabla \psi) =-\frac{1}{\Delta t} (\nabla \cdot \tilde{\uvec}_h^{n+1}, \psi), \\
      \end{aligned}
  \end{equation*}
and update $\uvec_h^{n+1}$ by
  $$ \uvec_h^{n+1} = \tilde{\uvec}_h^{n+1} - \Delta t \nabla (p_h^{n+1}-p_h^n).$$
\end{itemize}

\begin{remark}
In Step 1, the pressure-correction scheme has been adopted to decouple velocity  $\uvec^{n+1}$ and pressure $p^{n+1}$ in Eq. \eqref{finalmodelnondim}. 
Although the pressure-correction scheme can lead to an invertible discrete linear system at each time step regardless of whether the inf-sup condition is satisfied, the inf-sup condition is crucial for the stability and convergence of the numerical scheme \cite{Guermond2006}. 
Hence, in the current study, the inf-sup
stable Iso-P2/P1 element is adopted.
\end{remark}

Next we discuss how to solve microscopic part (\ref{Eq_qi}) with a given $\uvec_h^{n+1}$ at each node $\x_{\alpha}$. One difficulty is that $\qvec_i$ is a function of $\x$ and $t$ due to the convection term $\uvec \cdot \nabla \qvec_i(\x, t)$. Many earlier numerical studies based on CONNFFESSIT algorithms either focus on the shear flows in which the convection term vanishes or ignores the convection term \cite{Jourdainanalysis1, ottinger1996}. To deal with the convection term in stochastic methods, two types of methods have been developed. One is to introduce a spatial-temporal discretization to $\qvec_i$, as used in Brownian Configuration Field method \cite{ du2005fene, ottinger1997brownian, XuSPH2014}. Another way is to use a Lagrangian viewpoint to compute the convection term \cite{Halin1998}.
In the current study, we use the idea of the second approach, and use 
an operator splitting approach to solve Eq. \eqref{Eq_qi}. Initially, we assign ensemble of particles $\{\qvec_{{\alpha},  i} \}_{i=1}^N$ to each node $\x_{\alpha}$ ($\alpha = 1, 2, \ldots N_x$). We assume that $f(\x, \qvec, 0)$ is spatially homogeneous, 
and use the same ensemble of initial samples at all $\x_{\alpha}$. Within the values $\uvec_h^{n+1}$, we solve the microscopic equation (\ref{Eq_qi}) by the following two steps:

\begin{itemize}
 
 \item Step 2.1: At each node $\x_{\alpha}$, solve Eq. \eqref{Eq_qinondim} without the convection term $\uvec^{n+1}_h \cdot \nabla \qvec_i$ by
    \begin{equation}\label{ODE_q_1}
    \begin{aligned}
        & \frac{1}{N}\frac{\qvec_i^{n+1, *}-\qvec_i^{n}}{\Delta t} =
    - \frac{1}{2{\rm Wi}} \frac{\delta \hat{\mathcal{F}}_h}{\delta \qvec_i}(\{\qvec_i^{n+1, *}\}_{i=1}^N), \\
        & \qvec^{n+1, **} = ({\bf I} + (\nabla \uvec_h^{n+1}) \Delta t)\qvec^{n+1, *},  
    \end{aligned}
    \end{equation}
where  $\hat{\mathcal{F}}_h [\{\qvec_i\}_{i=1}^N] =  \tfrac{1}{N} \sum_{i=1}^N \left[ \ln \left(\frac{1}{N} \sum_{j=1}^N K_h(\qvec_i - \qvec_j )\right)+\Psi(\qvec_i)\right] $ is the non-dimensionalized discrete microscopic free energy at each node. We omit the index $\alpha$ when it does not cause confusion.

 \item Step 2.2:  To deal with the convection term $\uvec \cdot \nabla \qvec_i$, we view each node $\x_{\alpha}$ as a Lagrangian particle, and update it according to the Eulerian velocity field $\uvec_h^{n+1}$ at each node
\begin{equation}
\tilde{\x}_{\alpha} = \x_{\alpha} + \Delta t  (\uvec_h^{n+1}|_{\x_{\alpha}}), \quad \alpha = 1, 2, \ldots N_x.
\end{equation}
Hence, $\{ \qvec_{\alpha, i}^{n+1, **} \}$ is an ensemble of samples at the new point $\tilde{\x}_{\alpha}$. To obtain $\qvec_{\alpha, i}^{n+1}$ at $\x_{\alpha}$, we use a linear interpolation 
 to get $\qvec_{\alpha, i}^{n+1}$ (at mesh with $\{ \x_{\alpha} \}$ being the set of nodes) from $\qvec_{\alpha, i}^{n+1, **}$ (at mesh with $\{ \tilde{\x}_{\alpha} \}$ being nodes) for each $i$.
\end{itemize}

An advantage of the above update-and-projection approach is that it doesn't require a spatial discretization on $\qvec_i(\x, t)$. 
Within the ensemble of particles $\{\qvec^{n+1}_{{\alpha},  i} \}_{i=1}^N$ on each node $\x_{\alpha}$, the updated values of the viscoelastic stress $\bm \tau_h^{n+1}$ at each node, denoted as $\{\bm \tau_{{\alpha}}^{n+1}\}_{\alpha = 1}^{N_x}$, can be obtained through the second equation of Eq. \eqref{finalmodelnondim}.  And then project them into the finite element space of $\bm \tau$, i.e. $V_{\bm \tau}$. To this end, 
we choose the projection operator $\mathcal{I}$, such that, for each component of the stress $\tau_{{\alpha}, l, k }^{n+1}$ with $l, k=1, \cdots d$, $\mathcal{I}(\{\tau_{{\alpha}, l, k}^{n+1}\}_{\alpha = 1}^{N_x}) := \sum_{\alpha=1}^{N_x} \tau_{{\alpha}, l, k}^{n+1} \phi_{{\alpha}}$, where $\{\phi_{{\alpha}}: \alpha =1, \cdots, N_x\}\subset S_h$ denotes the nodal basis for $S_h$.

\begin{remark}
The operator splitting approach has been widely used in many previous Fokker-Planck based numerical approaches for micro-macro models \cite{helzel2006multiscale}. One important reason is that the system admits a variational structure without the convention terms. Moreover, by separating the convection component, the particles at each physical location can be treated independently, which largely saves the computational cost.
\end{remark}

Since the first step in (\ref{ODE_q_1}) admits a variational structure, the implicit Euler discretization can be reformulated as an optimization problem.
In more detail, at each node $\x_{\alpha}$, we define
$$
J_n[\{\qvec_i\}_{i=1}^N]=\frac{1}{N}\sum_{i=1}^N\left(\frac{1}{2\Delta t} |\qvec_i-\qvec_i^n|^2
\right)+ \frac{1}{2{\rm Wi}} \hat{\mathcal{F}}_h[\{\qvec_i\}_{i=1}^N],
$$
where $\hat{\mathcal{F}}_h [\{\qvec_i\}_{i=1}^N] =  \tfrac{1}{N} \sum_{i=1}^N \left[ \ln \left(\frac{1}{N} \sum_{j=1}^N K_h(\qvec_i - \qvec_j )\right)+\Psi(\qvec_i)\right] $.
We can obtain a solution to the nonlinear system by solving an optimization problem \cite{wangEVI, rockafellar1976monotone}
\begin{equation}\label{optimiproblem}
    \{\qvec_i^{n+1,*}\}_{i=1}^N= \mathop{\arg\min}_{\{\qvec_i\}_{i=1}^N} J_n(\{\qvec_i\}_{i=1}^N)\ .
\end{equation}
An advantage of this reformulation is that we can prove the existence of the $\qvec_i^{n+1,*}$. More precisely, we have the following result. 

\begin{proposition}\label{convergence}
For any given $\{\qvec_i^n\}_{i=1}^N$ at $\x_{\alpha}$, there exists at least one minimal solution 
$\{\qvec_i^{n+1}\}_{i=1}^N$ of \eqref{optimiproblem} that also satisfies 
\begin{equation}\label{energystable}
    \frac{1}{2{\rm Wi}}\frac{\hat{\mathcal{F}}_h(
    \{\qvec_i^{n+1}\}_{i=1}^N)- \hat{\mathcal{F}}_h(\{\qvec_i^{n}\}_{i=1}^N)}{\Delta t}\leq -\frac{1}{2N \Delta t^2}\sum_{i=1}^N|\qvec_i^{n+1}-\qvec_i^{n}|^2.
\end{equation}
\end{proposition}

\begin{proof}
Let $\bm X \in \mathbb{R}^D$ ($D=N\times d$) be vectorized $\{\qvec_i\}_{i=1}^N$, namely, $$\bm X=(q_{1,1}, \ldots,  q_{N,1}, q_{1,2}, \ldots, q_{N,2},\ldots, q_{N,d}).$$
Denote $\hat{\mathcal{F}}_h(\{\qvec_i\}_{i=1}^N)$ and $J_n(\{\qvec_i\}_{i=1}^N)$ as $\hat{\mathcal{F}}_h(\bm X)$ and $J_n(\bm X)$ respectively. For given $\bm X^n=\{\qvec_i^n\}_{i=1}^N$, we define
$$
S=\{J_n(\bm X)\leq J_n(\bm X^n)\}
$$
be the admissible set.
Obviously, $S$ is non-empty and closed, since $\bm X^n \in S$ and $J_n(\bm X)$ is continuous. Moreover, it's easy to prove that $\hat{\mathcal{F}}_h(\bm X)$ is bounded from below, since
$$
\ln\left(\frac{1}{N}\sum_{j=1}^N K_h(\qvec_i^{n}, \qvec_j^{n})\right)\geq \ln\left(\frac{1}{N} \frac{1}{(\sqrt{2\pi}h)^d} \right) \quad \mbox{and} \quad \Psi(\qvec_i)\geq 0.
$$
And thus, $J_n(\bm X)$ is  coercive
and $S$ is a bounded set. Hence, $J_n(\bm X)$ admits a global minimizer $\bm X^{n+1}$ in $S$. 
Therefore, we have
\begin{equation}
    \frac{1}{N}\sum_{i=1}^N\left(\frac{1}{2\Delta t} |\qvec_i^{n+1}-\qvec_i^n|^2
\right)+\frac{1}{2{\rm Wi}}\hat{\mathcal{F}}_h(\bm X^{n+1}) \leq \frac{1}{2{\rm Wi}} \hat{\mathcal{F}}_h(\bm X^{n}),
\end{equation}
which is equivalent to Eq.  \eqref{energystable}.
\end{proof}

In the numerical implementation, we adopt the Barzilai-Borwein gradient method \cite{barzilai1988two} to solve the optimization problem (\ref{optimiproblem}). The Barzilai-Borwein method is a gradient descent algorithm with an adaptive stepsize, which updates ${\bm X}^k$ (vectorized $\{\qvec_i\}_{i=1}^N$) through 
\begin{equation}
{\bm X}^k  = {\bm X}^{k-1} - \alpha_{k-1} \nabla_{\bm X} J_n( {\bm X}^{k-1}),
\end{equation}
where $\alpha_{k-1} = {\bm s}_k^{\rm T} {\bm s}_k / ({\bm s}_k^{\rm T} {\bm y}_k)$ is a BB stepsize. Here, ${\bm y}_k = \nabla_{\bm X} J( {\bm X}^{k-1}) - \nabla J_{\bm X} ( {\bm X}^{k-2})$ and ${\bm s}_k = {\bm X}^{k-1} - {\bm X}^{k-2}$. We compute ${\bm X}^1$ by using gradient descent with stepsize $10^{-7}$. 
The initial guess, ${\bm X}^0$, is taken as vectorized $\{\qvec_i^0\}_{i=1}^N$. The stopping criteria are set to $\| \nabla_{\bm X} J ({\bm X}) \|_2 \leq 10^{-9}$ or reaching the maximum number of inner iterations, which is set to 50.

\begin{remark}
In the current numerical scheme, we estimate the macroscopic stress tensor by taking microscopic distribution function as the empirical measure for the finite number of particles $\{ \qvec_i\}_{i=1}^N$. More advanced techniques can be applied to this stage to obtain a more accurate estimation to the stress tensor, such as the maximum-entropy based algorithm developed in Ref. \cite{arroyo2006local} and Ref. \cite{rosolen2013adaptive} that reconstructs basis functions from particles. We'll explore this perspective in future works.
\end{remark}

\section{Results and discussion} 

In this section, we perform various numerical experiments to validate the proposed numerical scheme by studying various well-known benchmark problems for the micro-macro models \cite{Laso1993, ottinger1996, Hyon2014}.

\begin{figure}[!htpb]
  \centering
  \includegraphics*[width = 0.9 \linewidth]{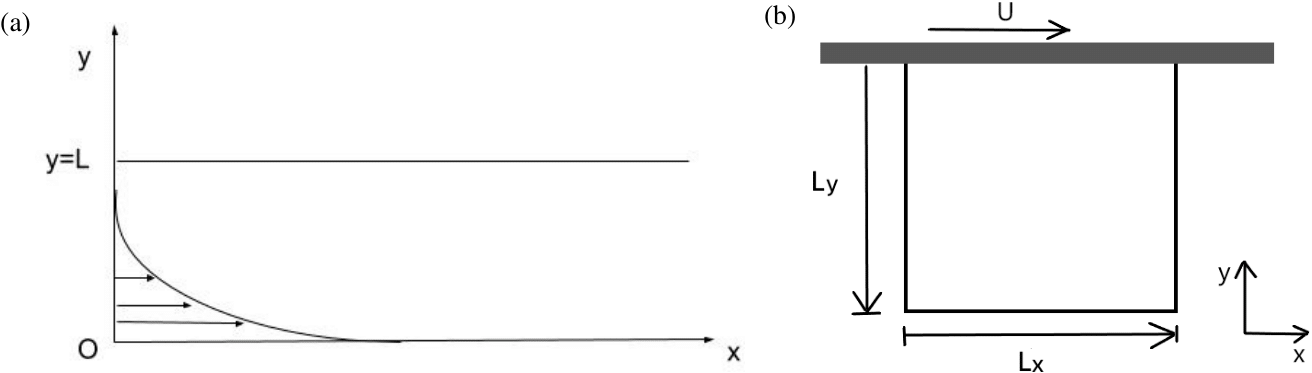}
  \caption{Schematic representation of (a) the initial simple shear flow and (b) the lid-driven cavity flow.}
  \label{two_type_flows}
  \end{figure}

We consider two flow scenarios: a simple shear flow and a lid-driven cavity flow. In a simple shear flow, a viscous fluid is enclosed between two parallel planes of infinite length, separated by a distance $L$, see Figure \ref{two_type_flows}(a) for an illustration. At $t=0$, the lower plane moves in the positive $\x$ direction with a constant velocity $U$. 
Due to the special geometry of simple shear flows, it is common to use the ansatz  $\uvec(\x,t)=(u(y,t),0)$, which means the velocity field $\uvec(\x,t)$ is in the $\x$-direction and depends only on the $y$-variable.
Additionally, we use the ansatz that $\qvec_i$ depends only on $y$-variable, which implies that $\uvec \cdot \nabla \qvec_i = 0$ \cite{Keunings1997, Laso1993}. Hence, the micro-macro model can be simplified into:
\begin{equation}\label{shearmodel}
\left\{
\begin{aligned}
& {\rm Re} \frac{\partial u}{\partial t}(y,t)= \tilde{\eta}_s \frac{\partial^2 u}{\partial y^2}(y,t)+\frac{\partial \tau_{21}}{\partial y}(y,t),\\
&  \tau_{21}=\frac{\epsilon_p}{{\rm Wi}} \frac{1}{N} \sum_{i=1}^N 
\bigg( \frac{\sum_{j=1}^N  [\nabla_{\qvec_i} K_h(\qvec_i, \qvec_j)]_1}{\sum_{j = 1}^N K_h(\qvec_i, \qvec_j) } + \sum_{k=1}^N \frac{[\nabla_{\qvec_i} K_h(\qvec_k, \qvec_i)]_1}{\sum_{j=1}^N K_h(\qvec_k, \qvec_j)} + [\nabla_{\qvec_i} \Psi(\qvec_i)]_1 \bigg) q_{i2}(y,t), \\
&\frac{\partial q_{i1}}{\partial t}(y,t)-\frac{\partial u}{\partial y}q_{i2}(y,t)=- \frac{1}{2{\rm Wi}} \left(\bigg[ \frac{\sum_{j=1}^N  [\nabla_{\qvec_i} K_h(\qvec_i, \qvec_j)]_1}{\sum_{j = 1}^N K_h(\qvec_i, \qvec_j) } + \sum_{k=1}^N \frac{[\nabla_{\qvec_i} K_h(\qvec_k, \qvec_i)]_1}{\sum_{j=1}^N K_h(\qvec_k, \qvec_j)}\bigg] + [\nabla_{\qvec_i} \Psi(\qvec_i)]_1 \right),\\
& \frac{\partial q_{i2}}{\partial t}(y,t)=-\frac{1}{2 {\rm Wi}} \left( \bigg[ \frac{\sum_{j=1}^N  [\nabla_{\qvec_i} K_h(\qvec_i, \qvec_j)]_2}{\sum_{j = 1}^N K_h(\qvec_i, \qvec_j) } + \sum_{k=1}^N \frac{[\nabla_{\qvec_i} K_h(\qvec_k, \qvec_i)]_2}{\sum_{j=1}^N K_h(\qvec_k, \qvec_j)} \bigg]+ [\nabla_{\qvec_i} \Psi(\qvec_i)]_2 \right),
\end{aligned}
\right.
\end{equation}
where $\tau_{21}$ is the off-diagonal component of the extra-stress tensor $\bm \tau$, $\qvec_i=(q_{i1}, q_{i2})$ and
$[\bm{g}]_{l}$ denotes the $l$th component of the vector $\bm{g}$ ($l=1,2$).

In the lid-driven cavity flow, the polymeric fluid is bounded in a two-dimensional rectangular box of width $L_x$ and height $L_y$, and the fluid motion is induced by the translation of the upper wall at a velocity $U$. The width of the cavity is set to be $L_x=1$, and the horizontal velocity of the lid $u(x)=U=1$. The three other walls are stationary, and the boundary conditions applied to them are non-slip and impermeable (see Fig. \ref{two_type_flows}(b)). In this case, a full 2D Navier-Stokes equation needs to be solved.

For all the numerical experiments carried out in this section, we assume that the flow is two-dimensional and the dumbbells lie in the plane of the flow, namely, the configuration vector $\qvec$ is also two-dimensional. 
At each node, we use the same initial ensemble of $N$ particles, sampled from the 2-dimensional standard normal distribution.

\subsection{Hookean model: simple shear flow}

It is well-known that the micro-macro model \eqref{CPsystemwithf_1} with a Hookean potential $\Psi(\qvec) = \frac{1}{2} H |\qvec|^2$ is equivalent to a macroscopic viscoelastic model, the Oldroyd--B model. For the Oldroyd-B model, an analytical solution for the start-up of plane Couette flow in a 2D channel is available (see \cite{Laso1993, XuSPH2014}).
So we can test the accuracy of the proposed numerical scheme by comparing the simulation results of the micro-macro model with the analytical solutions of the corresponding Oldroyd-B model.

\begin{figure*}[!hbtp]
  \centering
   \begin{overpic}[width=0.45\linewidth]{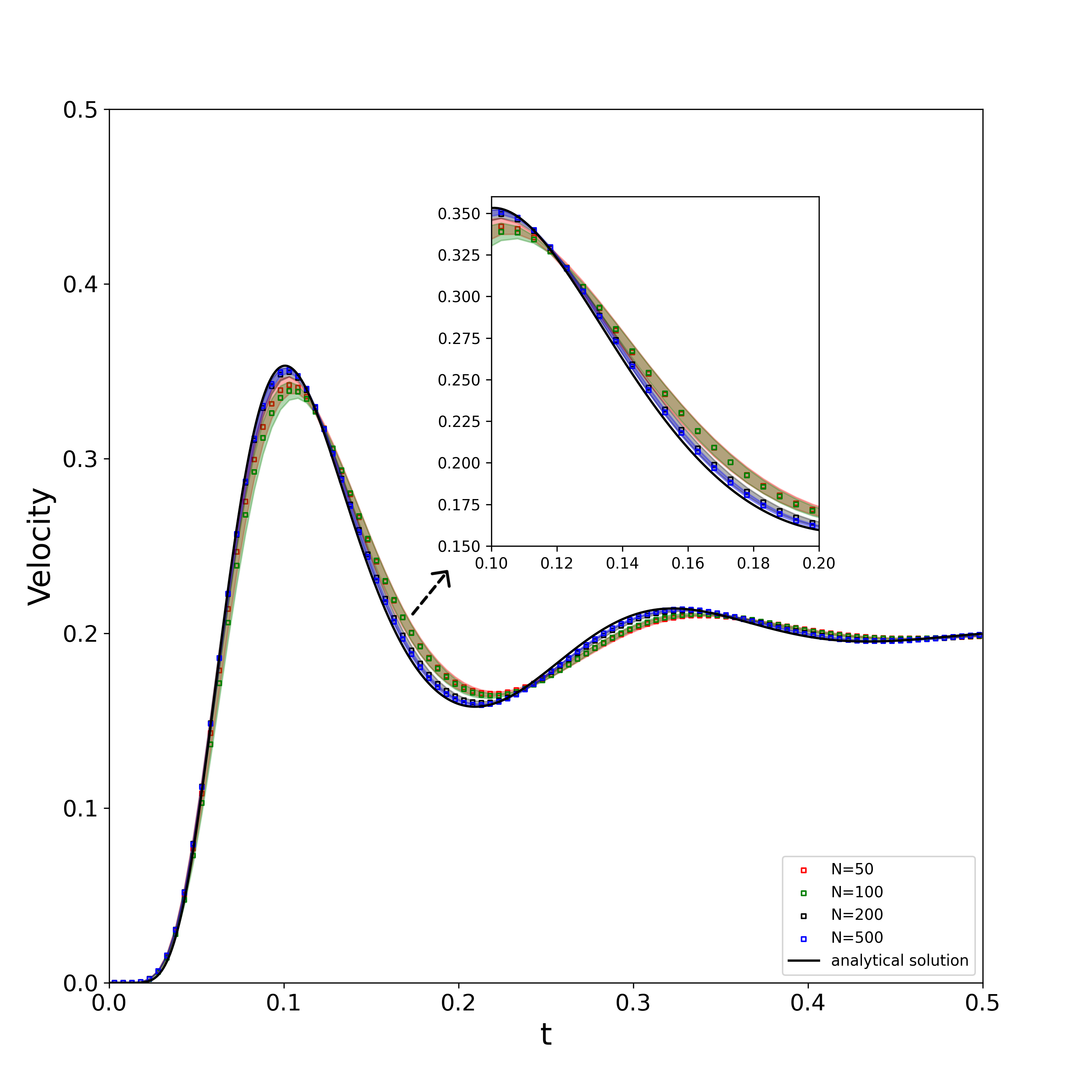}
  \put(-2, 85){(a)}
   \end{overpic}
   \begin{overpic}[width=0.45\linewidth]{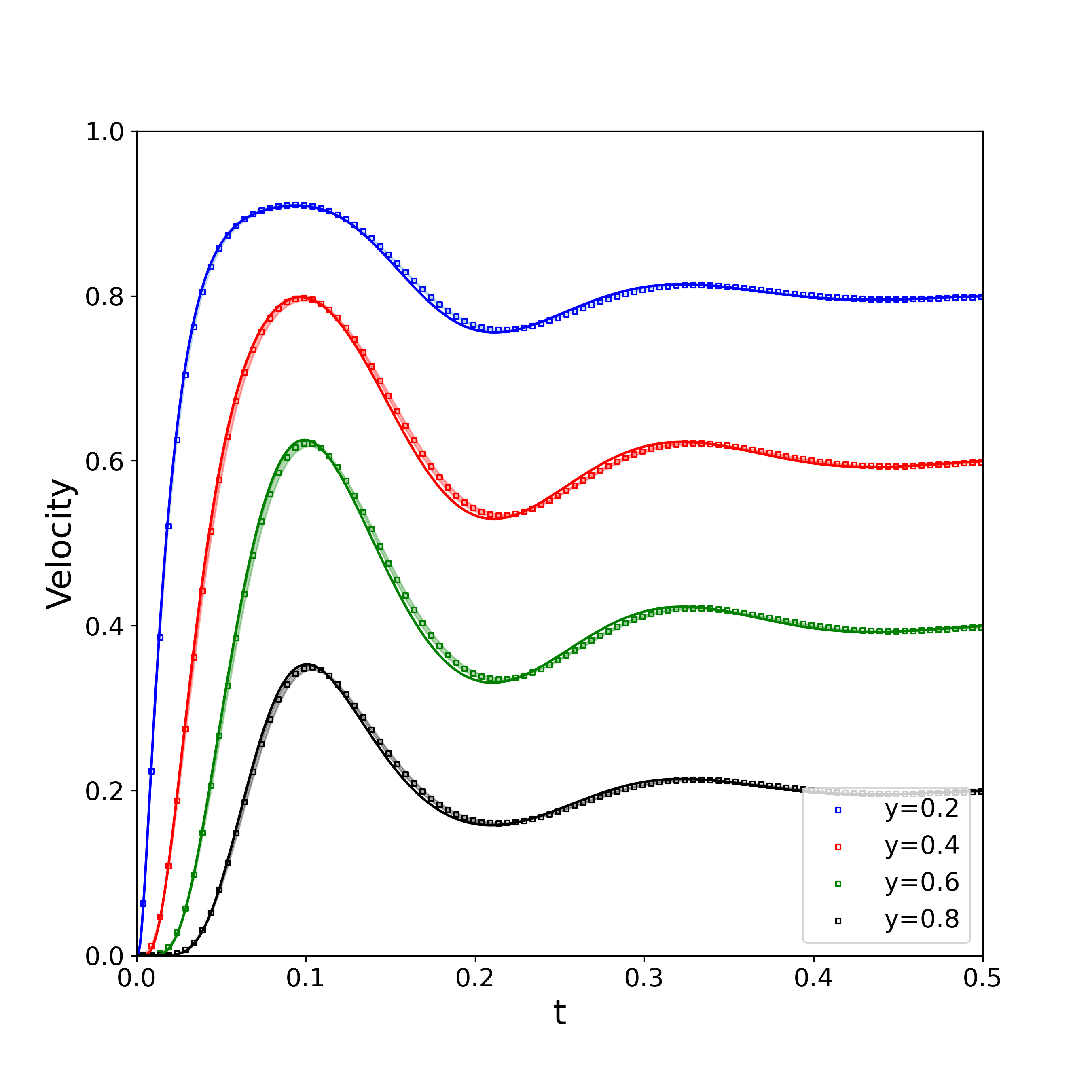}
    \put(-2, 85){(b)}
   \end{overpic}
  \caption{(a) Time evolution of velocity at $y=0.8$ with different numbers of particles. The analytical solution of the corresponding Oldrody-B model is shown in a black line.  (b) Comparison of the velocity of the Hookean case with particle number $N=200$ (marker) and the analytical solution (solid line) at different times and different locations ($y = 0.2, 0.4, 0.6, 0.8$).}
  \label{testaccuracy_hook}
\end{figure*}

We choose the physical parameters as follows: ${\rm Re} = 0.11$, ${\rm Wi} = 0.1$, $\tilde{\eta}_s = 0.11$ and $\epsilon_p = 0.89$. The number of elements is $M=40$ and the time step is $\Delta t =10^{-3}$. Additionally, we choose the kernel bandwidth as  $h = \mbox{med}_n/\sqrt{2 \log N}$, motivated by the median trick proposed in \cite{liu2016stein}, to compute $\{ \qvec^{n+1} \}_{i=1}^N$. Here, $\mbox{med}_n$ denotes the median of the pairwise distances between the particles $\{\qvec_i^n \}_{i=1}^N$.


We begin by examining the impact of the number of particles on numerical results. Figure \ref{testaccuracy_hook}. (a) shows the time evolution of velocity at $y = 0.8$ for different numbers of particles ($N = 50, 100, 200$ and $500$). we perform $10$ independent runs, where we use the same parameters and different sets of initial particles.  Each marker on the plot represents the mean values, indicating mean values ± standard errors \cite{Keunings1997}. The results show that as the number of particles increases, the mean values converge to the analytical solution and the standard errors decrease. We notice that, for $N =200$, the maximum standard error \cite{Keunings1997} and relative error \cite{Hulsen1997} in the velocity concerning time throughout all locations turns out to be approximately 0.008 and 9\%, respectively. 
Since a good numerical result can be achieved with $N = 200$ in this case, we set $N = 200$ in all the following numerical experiments for the efficiency of the numerical method.

Figure \ref{testaccuracy_hook} (b) presents the time evolution of the simulated velocity at $y = 0.2, 0.4, 0.6, 0.8$ for the micro-macro model with $N=200$ particles, compared to the analytical solution (solid line). 
It shows that our particle scheme with particle number $N=200$ yielded good numerical results.

\subsection{FENE model: simple extensional flows and hysteresis behavior}

The FENE models account for the finite extensibility of polymer chains, and can capture the hysteresis behavior of dilute polymer solutions in simple extensional flows during relaxation, which can be observed through the normal stress or elongational viscosity versus mean-square extension \cite{Lielens1998, Sizaire1999, Doyle1998}. However, many macroscopic closure models for the FENE potential fail to capture this behavior \cite{Sizaire1999, Hyon2014}. In this subsection, we demonstrate that the deterministic particle scheme is capable of capturing the hysteresis behavior of FENE models.

We consider the elongational velocity gradient given by \cite{Keunings1997}: 
\begin{equation}
\bm \kappa(t) 
= \varepsilon(t) \mbox{diag}(1, -1)\ ,
\end{equation}
where $\varepsilon(t)$ is the strain rate and $\mbox{diag}(1, -1)$ is the 2 $\times$ 2 diagonal matrix with diagonal entries being $1$ and $-1$.
Since the numerical simulations are carried out with a given velocity gradient $\nabla \uvec = \bm \kappa(t)$, the Fokker-Planck equation can be reduced
to
\begin{equation}\label{reduceFK}
    f_t+ \nabla_{\qvec}\cdot(\bm \kappa(t) \qvec f)= \frac{1}{2\rm Wi} \nabla_{\qvec}\cdot(f\nabla_{\qvec} \Psi )+  \frac{1}{2\rm Wi} \Delta_{\qvec} f, 
\end{equation}
where the distribution function $f = f(\qvec, t)$ is spatial homogeneous. The reduced Fokker-Planck equation \eqref{reduceFK} is employed for investigating the hysteresis behavior of the FENE model. The parameters in the model are taken as follows, ${\rm Wi} = 1; \   b = \sqrt{50}$.

Throughout this subsection, the initial data of particles is sampled from the 2-dimensional standard normal distribution. As stated in section 4.1, we set kernel bandwidth $h = \mbox{med}_n/\sqrt{2 \log N}$ for Hookean models, where $\mbox{med}_n$ is the median of the pairwise distance between the particles $\{\qvec_i^n\}_{i=1}^N$. 
However, the median trick is not suitable for the FENE potential, as the equilibrium distribution is no longer Gaussian type and the median of the pairwise distance can become very large. Numerical experiments show that taking kernel bandwidth $h = 0.01$ produces a good result for $N = 200$ for the FENE model in  simple extension flows. We also fix kernel bandwidth $h = 0.01$ and $N = 200$ for the following numerical experiments of the FENE models. We'll explore the effects of different kernel bandwidth $h$ in future work. The temporal step-size is taken as  $\Delta t = 10^{-3}$.

\begin{figure*}[!htbp]
  \centering
  \includegraphics[width = 0.9 \linewidth]{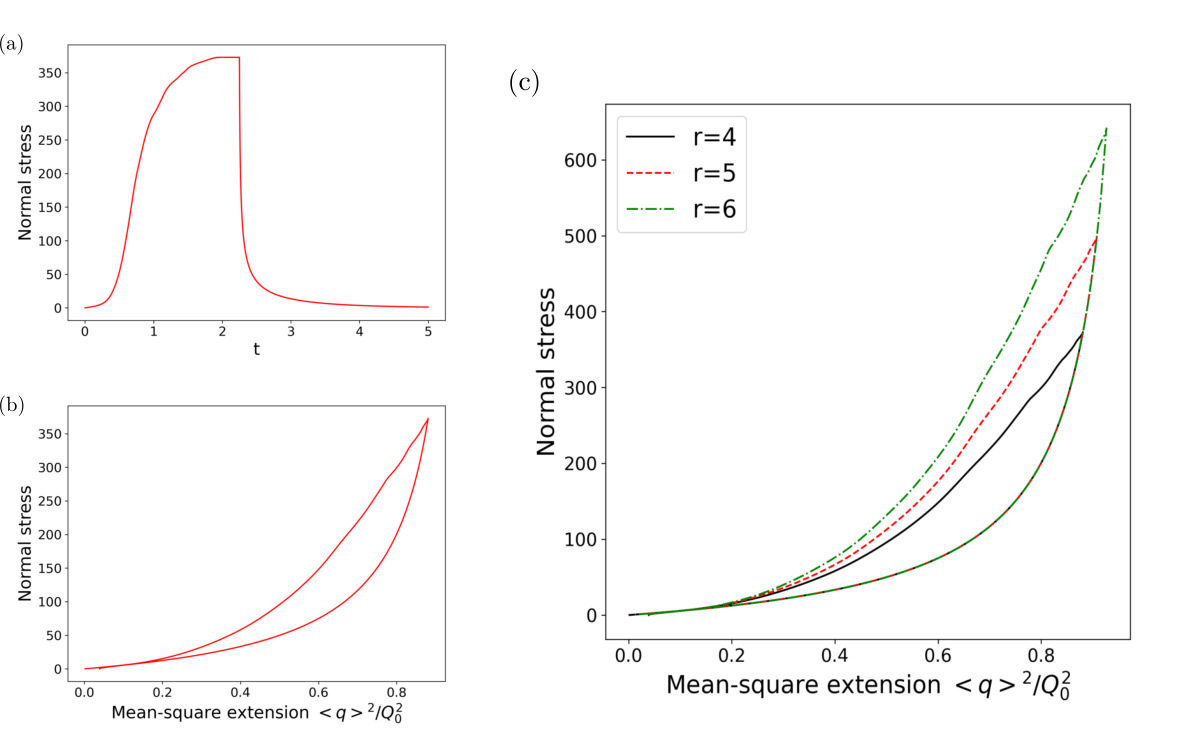}
  \caption{The start-up case with $r=4$: (a) the time evolution and (b) the hysteresis of normal stress $\tau_{11}-\tau_{22}$; (c) the hysteresis of the normal stress with $r=4$, $5$ and $6$.}\label{starupcase}
\end{figure*}

\begin{figure}[!htb]
\centering
\includegraphics*[width = 0.9 \linewidth]{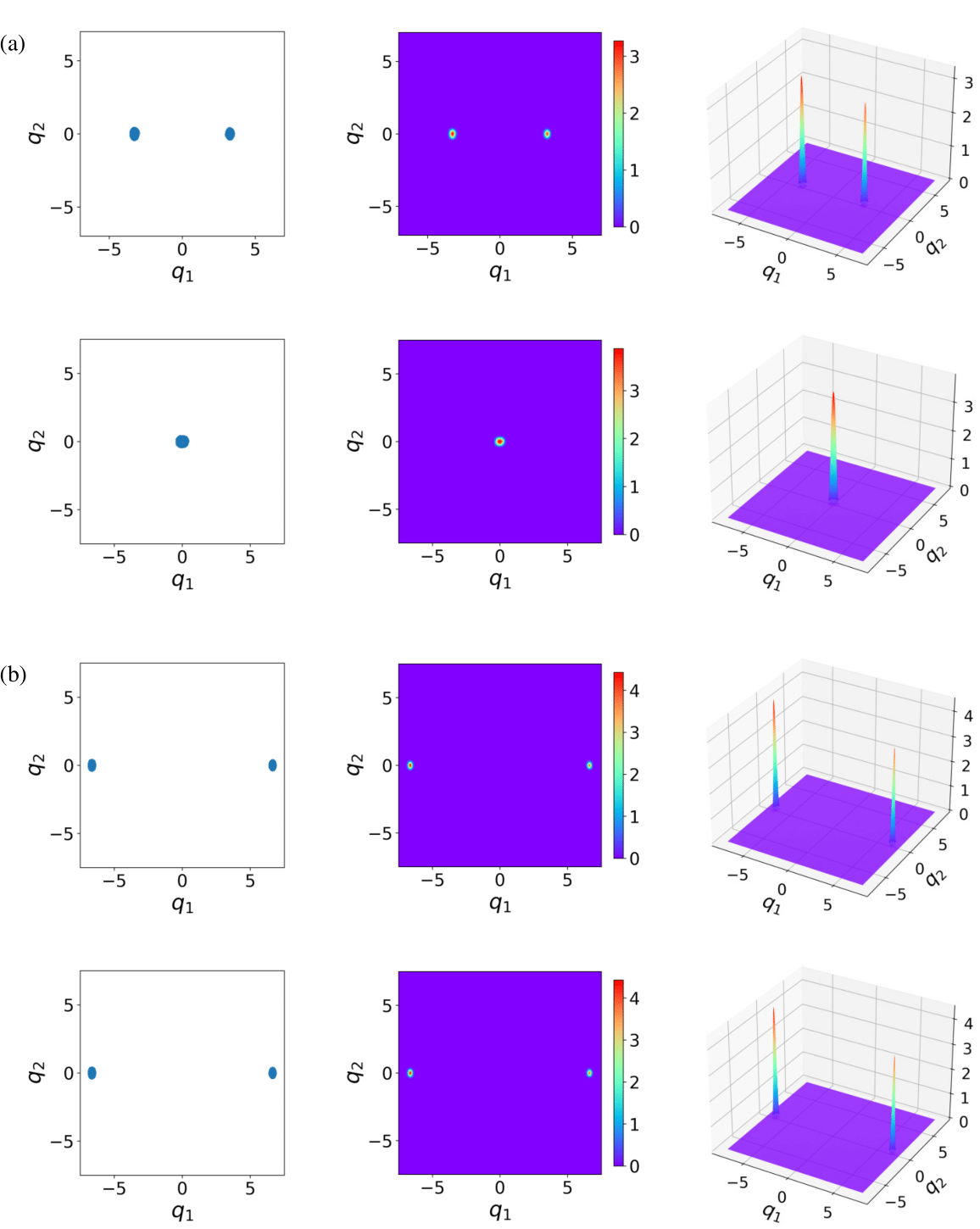}
\caption{The position of particles (left) and the underlying distribution of particles obtained by kernel density estimation  (middle and right) in the configuration space at different times with $r=4$. (a) Start-up case at $t=3$ (the first row) and $t=12$ (the second row). (b) The constant-gradient velocity case at $t=3$ (the first row) and $t=12$ (the second row).}\label{particlestarupcase}
 \end{figure}

We first consider the start-up case, in which $\varepsilon(t)$ 
given by
\begin{equation}\label{startepsilon}
\varepsilon(t)=
\left\{
\begin{aligned}
&r \qquad 0\leq t \leq \frac{9}r \ , \\
&0 \qquad \mbox{otherwise}\ .
\end{aligned}
\right.
\end{equation}
The time evolution of normal stress $\tau_{11}-\tau_{22}$ and the plot of the normal stress versus the mean-square extension $ \langle \qvec^2 \rangle /Q_0^2$ for $r = 4$ are plotted in Figure \ref{starupcase} (a)-(b). 
The comparison of hysteresis  behavior of the FENE model for different extensional flow rates ($r=4, 5, 6$) is shown in Figure \ref{starupcase} (c). It is observed that when the strength of velocity gradient is getting smaller, the hysteresis behavior becomes narrower. 
The numerical results are consistent with those obtained in the former work \cite{Hyon2014}.

As discussed in \cite{Hyon2014}, to catch the hysteresis of the original FENE model, a coarse-grained model should be able to catch the spike-like behaviors of the probability density in the FENE model in the large extensional effect of the flow field. The peak positions of the probability distribution function (PDF) distribution of the FENE model depend on the macroscopic flow field and change in time under the large macroscopic flow effects \cite{Hyon2008}. We show the position of particles (blue markers) and their underlying distribution probability density (obtained by the kernel density estimation) in the configuration space
at different times ($t = 3$ and $t = 12$) for the start-up case with $r=4$ in Figure \ref{particlestarupcase} (a).
The distribution splits into two spikes and then 
shows gradual centralized behavior. Eventually, it forms a single peak in the center, as shown in Figure \ref{particlestarupcase} (a). Notice that the numerical results in the equilibrium state are consistent with the equilibrium solution of the Fokker-Planck equation with zero flow rate, as the velocity rate turns to zero when $t$ is big enough ($t > 9/r$). 
We also consider a constant-gradient velocity case, in which $\varepsilon(t)= r$. In this case, the particles show two regions of higher concentration near the boundary of the configuration domain at the equilibrium state (i.e., with stable double spikes) \cite{Hyon2008}. Figure \ref{particlestarupcase} (b) shows the numerical results for $r = 4$ at $t = 3$ and $12$. The numerical results also indicate that the deterministic particle method can capture the $\delta$-function-like spike in the FENE model.

\subsection{FENE model: simple shear flow}

\begin{figure*}[!htb]
  \centering
   \includegraphics[width = 0.9 \linewidth]{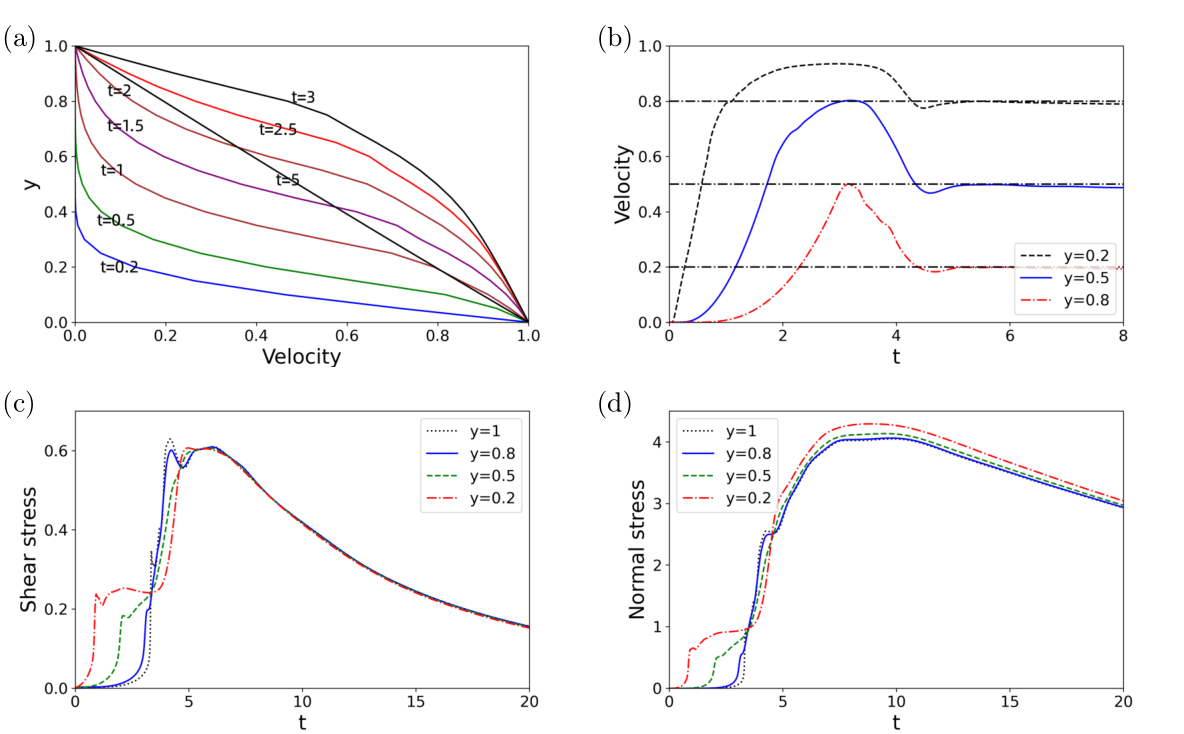}
  \caption{For the FENE model: the velocity $u$ with respect to location $y$ at different times (a); the time evolution of the velocity $u$ at different locations (b); the time evolution of the shear stress(c) and normal stress difference(d) at location $y=0.2$, $y=0.5$, $y=0.8$ and $y=1$.}\label{feneuyt}
\end{figure*}

In this subsection, we evaluate the proposed algorithm's performance for the micro-macro model with a FENE potential in a start-up plane Couette flow as shown in Figure 1(a). 
We set $L=1$, $M=20$, $\Delta t = 10^{-3}$.
The non-dimensional parameters are chosen to be ${\rm Re} = 1.2757$, $\tilde{\eta}_s=0.0521$, ${\rm Wi} = 49.62$, $\epsilon_p=0.9479$, and $b=\sqrt{50}$, which are the same as those used in Ref. \cite{Laso1993}.

Figure \ref{feneuyt} (a) shows the velocity evolution with respect to the location $y$ at different times. It reveals the velocity overshoot phenomenon for the FENE model, which is a typical property of viscoelastic fluids. Figure \ref{feneuyt} (b) displays the velocity evolution with respect to time $t$ at three locations $y=0.2$, $y=0.5$ and $y=0.8$. It can be seen that the velocity overshoot occurs sooner in fluid layers nearer to the moving plane. Figures \ref{feneuyt}(c-d) show the temporal evolution of the shear stress and the normal stress difference at different locations $y=0.2$, $y=0.5$, $y=0.8$, and $y=1$. The stress response is sharper in fluid layers nearer to the moving plane, which is consistent with the behavior of velocity overshoot. We observe that the maximum of the normal stress occurs after the maximum of the shear stress. Specifically, the shear stress of the FENE model reaches its maximum at around $t=6$, but the maximum of the normal stress is reached at about $t=10$. The numerical results agree well with the former work \cite{Laso1993, XuSPH2014}, indicating the accuracy of our numerical scheme in the FENE case. Moreover, compared with the former work, our numerical results obtained by the deterministic particle scheme show fewer oscillations.

\begin{figure*}[htpb]
   \centering
   \begin{overpic}[width=0.45\linewidth]{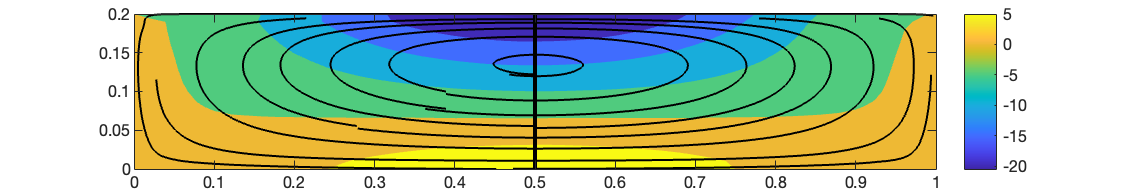}
  \put(0, 15){(a)}
   \end{overpic}
   \begin{overpic}[width=0.45\linewidth]{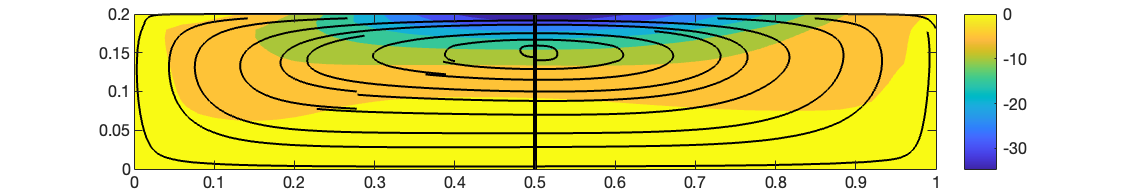}
    \put(0, 15){(b)}
   \end{overpic}
   
   \vspace{0.5cm}
   
   \begin{overpic}[width=0.45\linewidth]{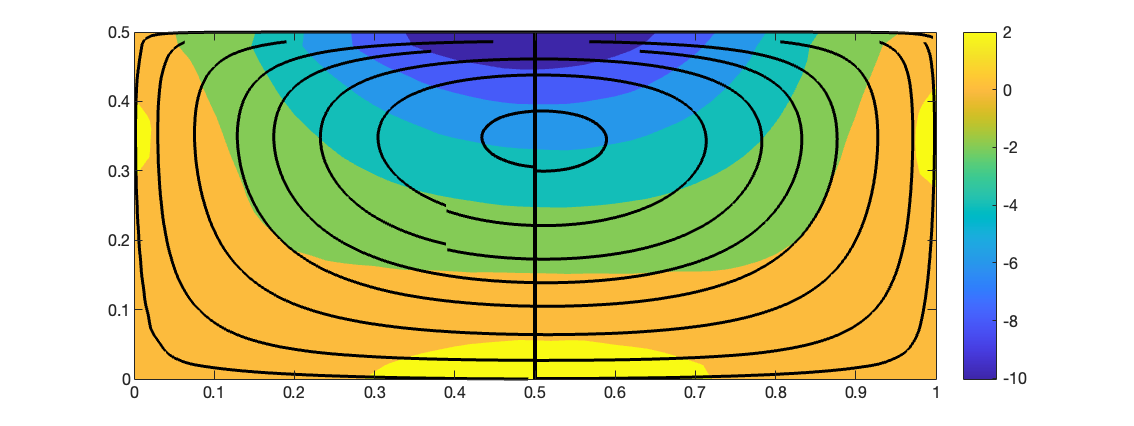}
      \put(0, 32){(c)}
  \end{overpic}
    \begin{overpic}[width=0.45\linewidth]{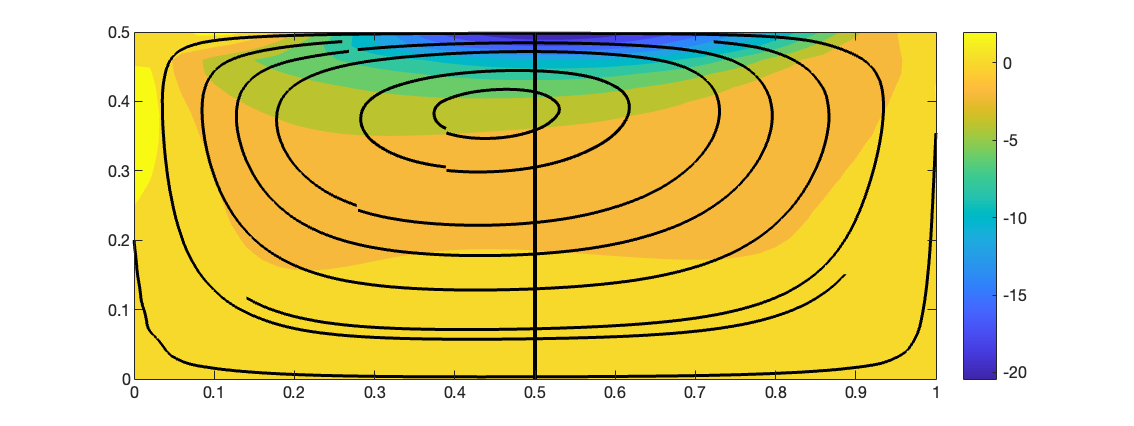}
    \put(0, 32){(d)}
   \end{overpic}
   
   \vspace{0.5cm}
   
   \begin{overpic}[width=0.45\linewidth]{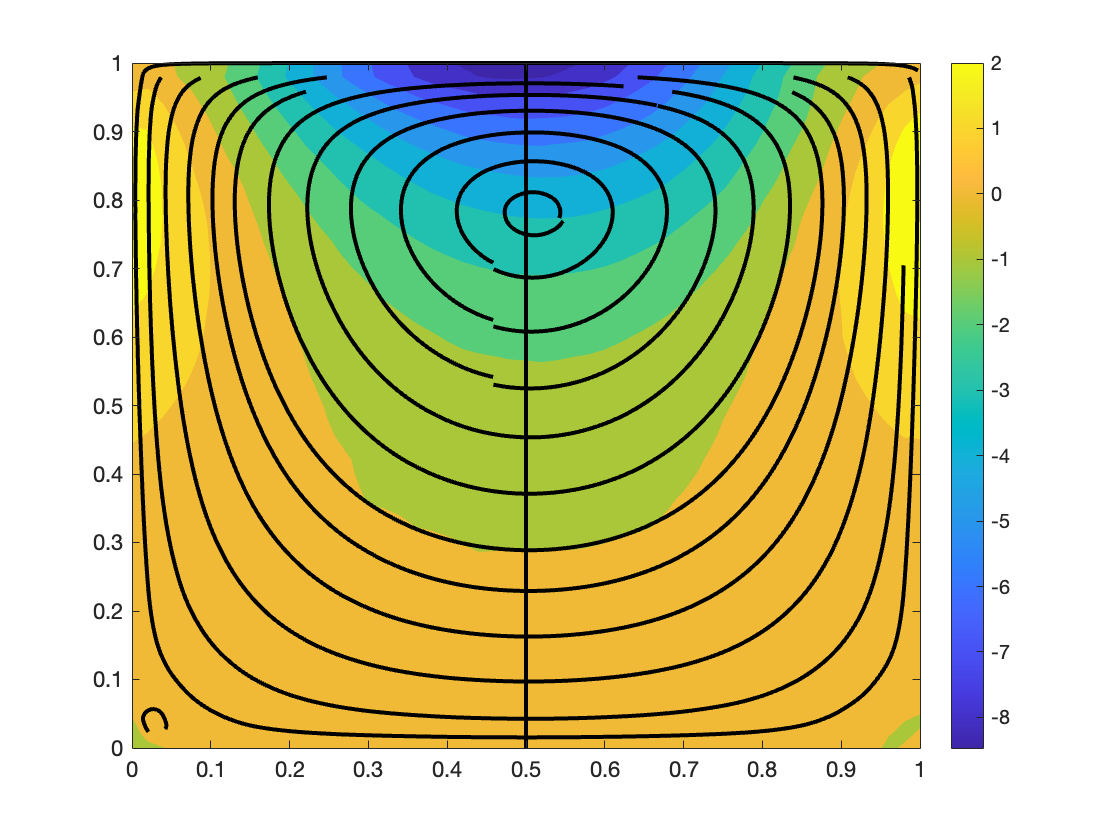}
      \put(0, 65){(e)}
  \end{overpic}
    \begin{overpic}[width=0.45\linewidth]{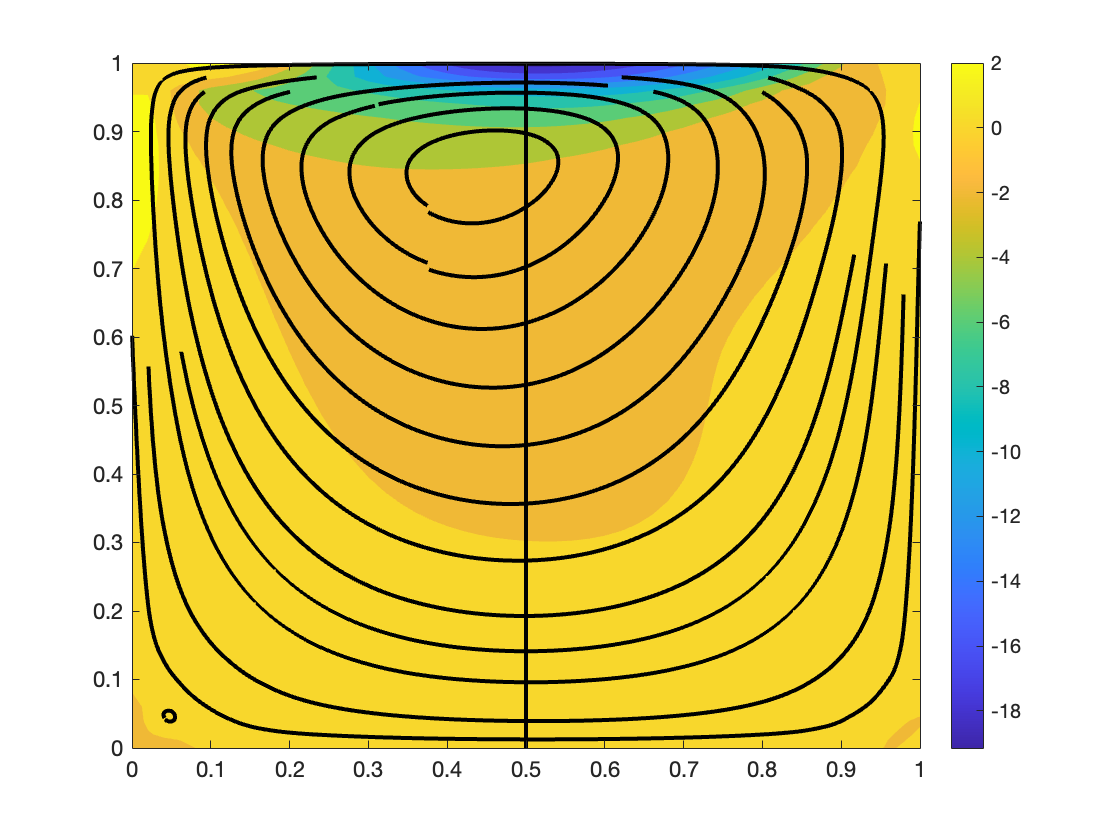}
    \put(0, 65){(f)}
   \end{overpic}
   
   \vspace{0.2cm}
   
   \begin{overpic}[width=0.3\linewidth]{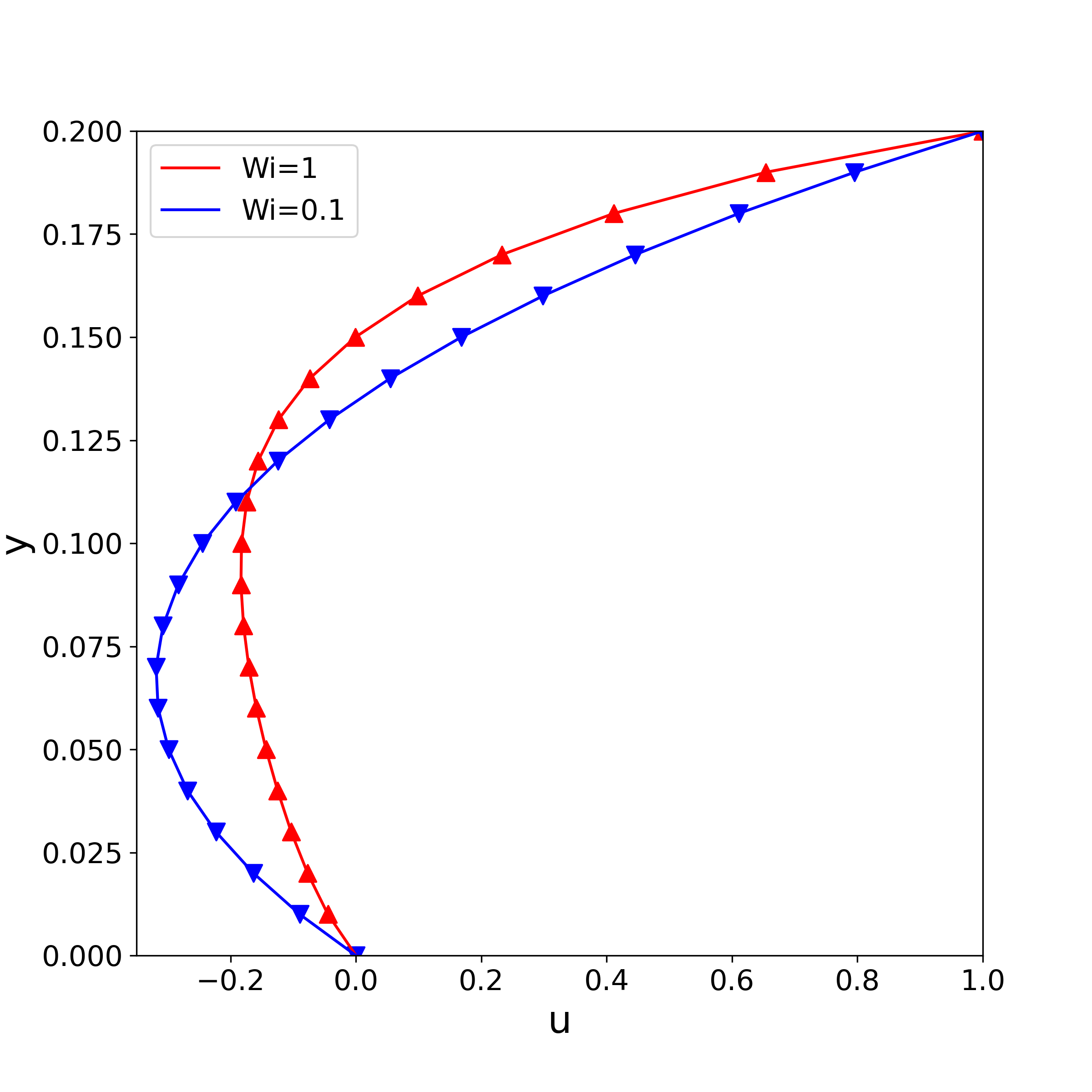}
      \put(-6, 85){(g)}
  \end{overpic}
    \begin{overpic}[width=0.3\linewidth]{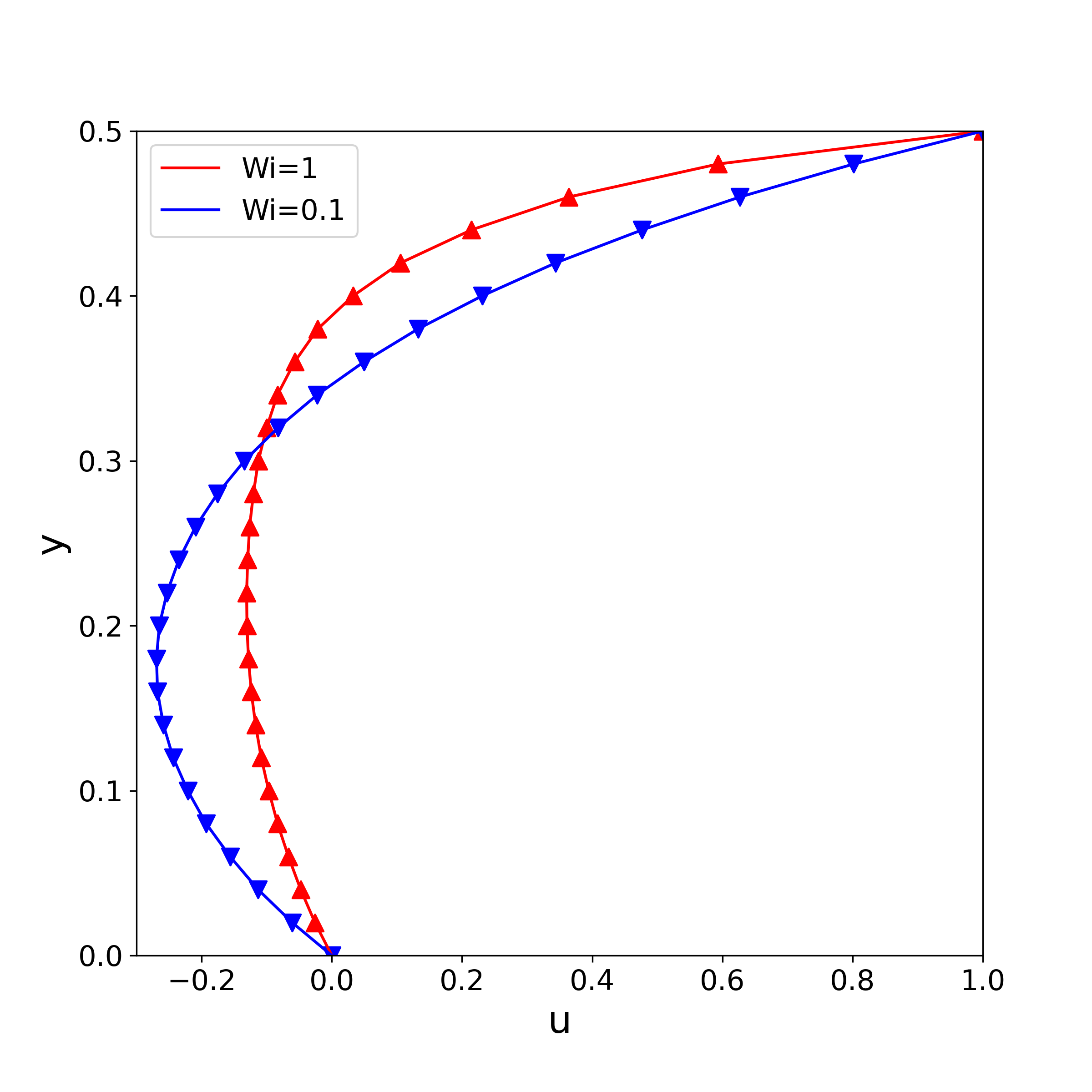}
    \put(-3, 85){(h)}
   \end{overpic}
       \begin{overpic}[width=0.3\linewidth]{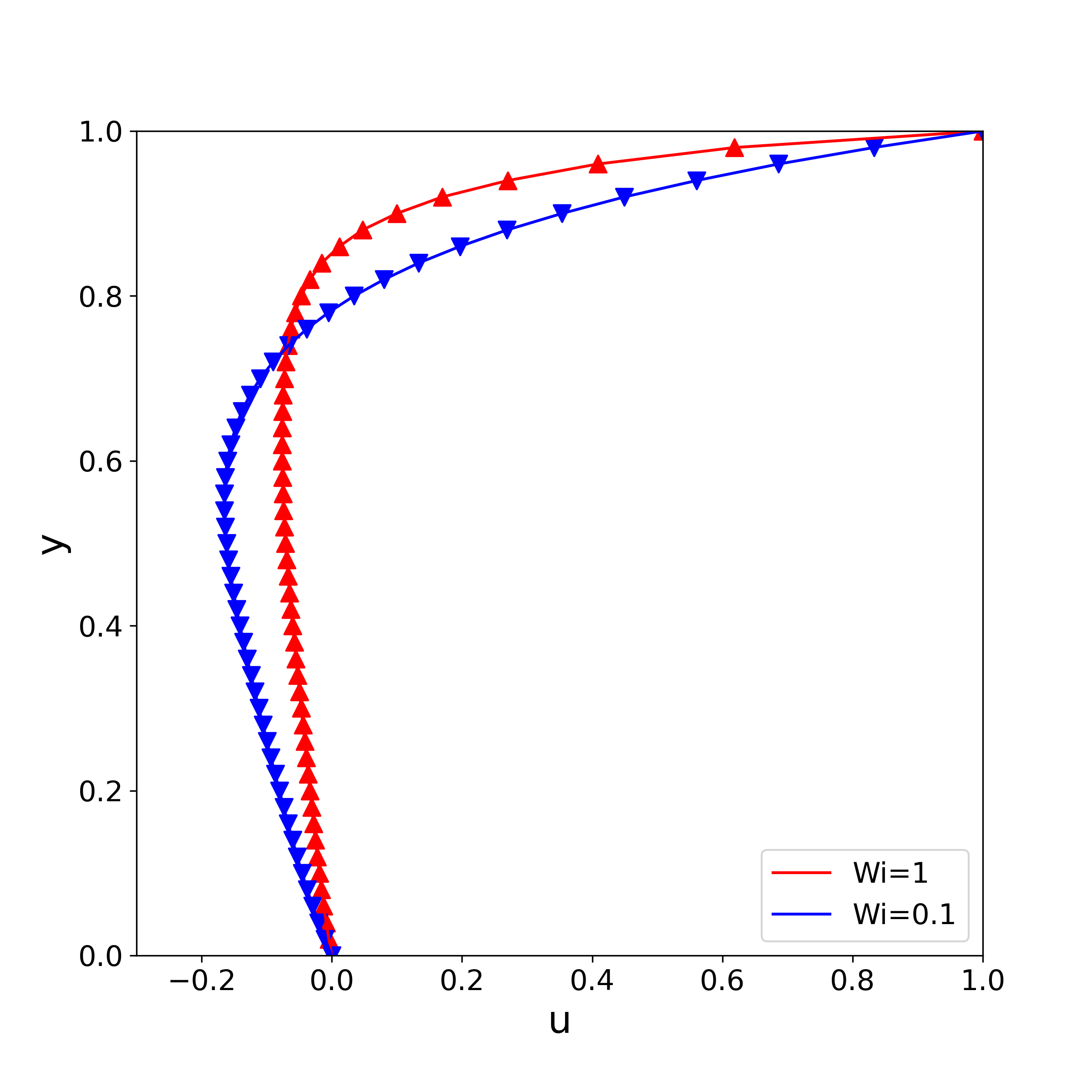}
    \put(-3, 85){(i)}
   \end{overpic}
  \caption{The streamlines and vortice contours of the lid-driven cavity flow in the FENE case at $T=1$ with $L_y=0.2$ ((a): ${\rm Wi} = 0.1$; (b): ${\rm Wi} = 1$), $L_y=0.5$ ((c): ${\rm Wi} = 0.1$; (d): ${\rm Wi} = 1$) and $L_y=1$ ((e): ${\rm Wi} = 0.1$; (f): ${\rm Wi} = 1$). The profile of $u$-velocity with respect to $y$ at position $x = 0.5$ for ${\rm Wi} = 0.1, 1$ with $L_y=0.2$ (g), $L_y=0.5$ (h) and $L_y=1$ (i).}\label{cavityflowu}
\end{figure*}

\subsection{FENE model: lid-driven cavity flow}

In this subsection, we simulate the FENE model for lid-driven cavity flows (see Figure \ref{two_type_flows}(b)). 
It is a 2D problem and a full 2D Navier--Stokes equation needs to be solved.
Our experiments consider rectangular cavities with different heights: $L_y=0.2$, $0.5$, and $1$. 
To avoid numerical difficulties that arise from the geometric singularity at the edges of an idealized lid-driven cavity, we adopt a regularized horizontal lid velocity \cite{Sousa16} of the form:
$$
u(x) = 16U(x/L_x)^2(1-x/L_x)^2.
$$

To discretize the problem spatially, we choose $\mathcal{T}_h$ to be a uniform triangular mesh with the mesh size $N_x=50, N_y=20$ for $L_y=0.2$, $N_x=50, N_y=25$ for $L_y=0.5$, and $N_x=50, N_y=50$ for $L_y=1$.
%
We set the time step size $\Delta t = 10^{-3}$ for temporal discretization. Unlike the shear flow cases, the convection term $\uvec \cdot \nabla \qvec$ is non-zero, which is dealt with by a Lagrangian approach as introduced in section 3.
Other parameters in the numerical experiments are set as follows:
$
{\rm Re} = 1; \  \tilde{\eta}_s=0.11; \  
\epsilon_p =0.889; \  b = \sqrt{50}. 
$ 

Figures \ref{cavityflowu}(a)-(f) display the streamlines and vortice contours for different $L_y$ at time $T=1$ with ${\rm Wi}=0.1$ and ${\rm Wi} = 1$. 
Notice that the streamlines show symmetry structure when ${\rm Wi}=0.1$. And this symmetry structure holds for different $L_y$. 
However, as elasticity becomes more important, namely, the Weissenberg number (${\rm Wi}$) increases, the symmetries in the streamline structures break due to the presence of elastic effects \cite{rectangularcavity16}.  
Meanwhile, as the flow becomes asymmetric, the vortex center in the cavity shifts progressively upward and opposite to the direction of lid motion \cite{Anne1999}. This phenomenon is more evident in Figures \ref{cavityflowu}(g)-(i), which plot the $u$-velocity profiles at $x=0.5$ for the cavity flow with different $L_y$ and ${\rm Wi}$.
Additionally, the introduction of elasticity also weakens the strengths of vortices near the moving lid \cite{Sousa16}. 

The numerical results indicate that the particle-based scheme can capture these complex behaviors.
The qualitative agreement between our simulation results and those of the former work \cite{XuSPH2014, Anne1999} validates our numerical scheme for the 2D lid-driven cavity flow case.

\section{Conclusion}

In this article, we present a novel deterministic particle-finite element method (FEM) discretization for micro-macro models of dilute polymeric fluids. The proposed scheme employs a finite element method for the fluid flow equation and a variational particle scheme used for the kinetic viscoelastic model. 
The proposed scheme is validated through various benchmark problems, including steady flow, shear flow, and 2D lid-driven cavity flow.
Our numerical results are in excellent agreement with those from the former work \cite{Laso1993, XuSPH2014, Anne1999,  Hyon2008, Hyon2014} and demonstrate that the proposed scheme can capture certain complex behaviors of the nonlinear FENE model, including the hysteresis and $\delta$-function like spike behavior in extension flows \cite{Hyon2008, Hyon2014}, velocity overshoot phenomenon in pure shear flow \cite{Laso1993},  
symmetries breaking, vortex center shifting and vortices weakening in lid-driven cavity flow \cite{XuSPH2014}. Compared with the stochastic simulation methods in the former work \cite{Hulsen1997, Laso1993, XuSPH2014}, where a large ensemble of realizations of the stochastic process is needed, the deterministic particle scheme can achieve good numerical results with reduced oscillations using only a small number of particles.

The proposed method can also be applied to other complex fluid models, such as the Doi-Onsager model for liquid crystal polymers \cite{doi1988theory}, the multi-bead spring model \cite{zhou2004MBE}, and a two-species model for wormlike micellar solutions \cite{liu2021two, wang2021two}, which involves a reaction in the microscopic equation. Additionally, as a direction for future work, 
we aim to develop an energy-stable scheme for the overall system.

\section*{Acknowledgement}
C. Liu is partially supported by NSF DMS-1950868, DMS-2153029, and DMS-2410742.
Y. Wang is partially supported by NSF DMS-2153029 and DMS-2410740. 
Part of this work was done when X. Bao visited the Department of Applied Mathematics at the Illinois Institute of Technology,  she would like to
acknowledge the hospitality of IIT. 
X. Bao also would like to thank Prof. Rui Chen for the helpful discussions.

\section*{Declaration of generative AI and AI-assisted technologies in the writing process}
During the preparation of this work, the author(s) used ChatGPT in order to improve Language only. After using this tool/service, the author(s) reviewed and edited the content as needed and take(s) full responsibility for the content of the publication.

\bibliographystyle{elsarticle-num} 
\bibliography{MM}

\end{document}